\documentclass{amsart}

\usepackage{amsmath,amssymb,amsfonts,mathtools}
\usepackage{enumitem}
\usepackage{microtype}
\usepackage{cite}
\usepackage[hidelinks]{hyperref}

\numberwithin{equation}{section}

\newtheorem{theorem}{Theorem}[section]
\newtheorem{lemma}[theorem]{Lemma}
\newtheorem{proposition}[theorem]{Proposition}
\newtheorem{corollary}[theorem]{Corollary}
\theoremstyle{definition}
\newtheorem{definition}[theorem]{Definition}
\theoremstyle{remark}
\newtheorem{remark}[theorem]{Remark}
\newtheorem{example}[theorem]{Example}

\newcommand{\R}{\mathbb{R}}
\newcommand{\defeq}{\triangleq}

\title[Fundamental Limits of Adaptive Stabilization]{Fundamental Limits of Adaptive Stabilization with an Unknown Growth Exponent}
\author{Zhaobo Liu}
\address{Institute for Advanced Study, Shenzhen University, Shenzhen, Guangdong 518060, China}
\email{liuzhaobo@szu.edu.cn}
\thanks{This work was supported in part by the National Natural Science Foundation of China under Grant 12401585, the Guangdong Basic and Applied Basic Research Foundation under Grant 2024A1515011542, and the General Program of the Shenzhen Natural Science Foundation under Grant JCYJ20250604181037012.}
\date{}

\hypersetup{
  pdftitle={Fundamental Limits of Adaptive Stabilization with an Unknown Growth Exponent},
  pdfauthor={Zhaobo Liu}
}

\begin{document}

\begin{abstract}
A basic question in adaptive control is how rapidly a discrete-time nonlinear
plant with unknown parameters may grow while remaining stabilizable. When the
nonlinear growth exponent is known and only a scalar coefficient is unknown,
existing theory identifies $4$ as the exact critical exponent for
stabilizability. We determine how this critical exponent changes when the
growth exponent is also unknown. For any positive disturbance bound, one
feedback law stabilizes all plants in some neighborhood of a nominal
parameter pair if and only if the nominal exponent is below
$3\sqrt{3}/2$. At equality, every compact parameter set whose exponents do not
exceed this value remains stabilizable. If the exponent belongs to a known
finite set, the critical exponent remains $4$ for every nondegenerate compact
coefficient interval. Hence finite exponent sets and arbitrarily short
exponent intervals can have different critical exponents. Uncertainty in the
growth rate therefore changes the range of growth that feedback can
stabilize, not merely the size of the uncertainty set. The critical exponent
remains $3\sqrt{3}/2$ under a known positive state-dependent multiplier bounded
above and away from zero. For a system whose coefficient and exponent are known
functions of an unknown parameter, a compact parameter family is stabilizable
when its exponents do not exceed this value. At an interior parameter point
where the Jacobian of these functions has rank two and the exponent is at
least this value, no compact neighborhood is stabilizable.
\end{abstract}

\maketitle

\section{Introduction}\label{sec:introduction}

A central question in adaptive control is whether feedback can stabilize a
plant whose parameters are not known in advance
\cite{AstromWittenmark1995,ChenGuo1991,IoannouSun1996}. For scalar
discrete-time systems with superlinear drift, stabilizability can depend
sharply on the growth exponent. When the exponent is known and only a scalar
coefficient is unknown, the exact critical exponent is $4$. This value was
first obtained under stochastic noise \cite{Guo1997} and was later shown to
be exact under bounded disturbances \cite{LiXie2006}. In the respective
settings, adaptive stabilization is possible below $4$ and impossible at or
above $4$. Related work studies the capability and limitations of feedback for
broader classes of uncertain nonlinear systems \cite{Guo2002,Guo2020}. The
critical exponent $4$, however, presumes that the controller knows the growth
exponent.

Knowing the exponent fixes the state nonlinearity and leaves only its
coefficient to be learned. If the exponent is also unknown, the controller
must infer from the closed-loop trajectory both the scale of the drift and how
rapidly it changes with the state magnitude. Does uncertainty in the exponent
merely make adaptation more difficult, or does it change the range of nonlinear
growth that feedback can stabilize? This paper determines the corresponding
critical exponent when both the coefficient and the exponent are unknown.

We consider
\begin{equation}\label{eq:system}
 y_{t+1}=a\phi_b(y_t)+u_t+w_{t+1},
 \qquad
 \phi_b(y)\defeq \operatorname{sgn}(y)|y|^b,
\end{equation}
where $y_t\in\R$ is the state, $u_t\in\R$ is the control input, both the
coefficient $a\in\R$ and the exponent $b>0$ are unknown, and the disturbance
satisfies $|w_t|\le W$ for a fixed known $W>0$. The function $\phi_b$ agrees
with $y^b$ for $y\ge0$ and is used so that noninteger exponents are defined on
the whole real line. The controller knows the prescribed parameter set and
chooses $u_t$ causally from the observed state history. We ask whether one
feedback law can keep the state bounded for every parameter in some
neighborhood of a nominal parameter pair, every initial state, and every
disturbance sequence satisfying the bound.

At a fixed nonzero state magnitude, a change in the coefficient can offset a
change in the exponent, so one observed transition cannot distinguish
parameter pairs that produce different drifts at larger magnitudes. The
controller must therefore use observations at different state magnitudes while
keeping the closed-loop state bounded.

A central line of work on exact feedback limits with several unknown parameters
assumes that the functions of the state are prescribed in advance.
Under bounded disturbances, \cite{LiXieGuo2006} treats linear combinations of
finitely many known functions whose polynomial growth rates are strictly
ordered. The same assumption appears in a related least-squares result under
Gaussian noise \cite{LiLam2013} and in later work allowing more general growth
\cite{LiuLi2019,LiuLi2023}. Related studies address several unknown
coefficients \cite{XieGuo1999}, an unknown control coefficient
\cite{LiGuo2010}, nonparametric and semiparametric uncertainty
\cite{XieGuo2000,HuangGuo2012}, and impossibility results based on dynamical
inequalities \cite{LiGuo2013}. None directly covers \eqref{eq:system}, because
the unknown exponent changes the state function itself rather than a
coefficient multiplying a prescribed function.

Work on nonlinear parameterizations is closer to the present setting.
Li and Guo \cite{LiGuo2011} assume that the derivatives of the drift
with respect to the unknown parameters have strictly ordered polynomial
growth rates. For several unknown parameters, their conditions for
stabilization and impossibility do not coincide. In \eqref{eq:system},
when $a\ne0$ and $y\ne0$, the derivatives with respect to $a$ and $b$
have the same leading power $|y|^b$, with an additional logarithmic
factor in the latter. The required ordering therefore fails. Other
results require these derivatives to grow at most linearly
\cite{LiChen2014} or treat a single unknown scalar parameter
\cite{LiChen2016,Li2018,LiuLiCTA2019}. Consequently, these results do
not determine the exact critical exponent for \eqref{eq:system}.

To resolve this difficulty, we compare parameters consistent with observations
made at different state magnitudes and use the resulting bounds to construct a
stabilizing feedback law. For the converse, we construct an escaping state path
together with a convergent sequence of auxiliary parameter values. The limit
specifies one fixed plant, and a fixed bounded disturbance sequence makes the
constructed path a trajectory of that plant. When both the coefficient and
the exponent are unknown, the exact critical exponent is $3\sqrt{3}/2$. Robust
stabilization over a parameter neighborhood is possible below this value and
impossible at or above it. When the nominal exponent equals $3\sqrt{3}/2$,
however, every compact parameter set whose exponents do not exceed this value
remains stabilizable.
When the exponent is known and only the coefficient is unknown, the
corresponding critical exponent is $4$. Thus allowing the exponent to vary over
an arbitrarily small interval changes the range of nonlinear growth that
feedback can stabilize, not merely the difficulty of estimating the plant.

The main contributions are as follows:
\begin{enumerate}[label=\textup{(\roman*)}]
\item The stabilization theorem covers every compact parameter set whose
largest exponent is at most $3\sqrt{3}/2$, including the endpoint, and gives a
common state bound over the parameter set, every prescribed bounded set of
initial states, and all disturbance sequences bounded in magnitude by $W$.
Conversely, no compact rectangle containing in its interior a point whose
exponent exceeds $3\sqrt{3}/2$ is robustly stabilizable. For every
feedback law, the converse selects one fixed plant in the rectangle, an
arbitrarily large positive initial state, and one fixed bounded disturbance
sequence such that the resulting trajectory escapes.

\item When the exponent belongs to a known finite set, robust stabilization
over every nondegenerate compact coefficient interval is possible if and only
if the largest candidate exponent is below $4$. Finite exponent sets of
arbitrarily small positive diameter can therefore be stabilizable even when
their interval hulls are not. The difference between finite and continuous
exponent uncertainty is thus not explained by the size of the uncertainty set
alone.

\item  The critical exponent for continuous exponent uncertainty remains
$3\sqrt{3}/2$ when the nonlinear term in \eqref{eq:system} is multiplied by a
known positive function of the state that is bounded above and away from zero.
For a known nonlinear parameter map, stabilizability depends only on the image
of the parameter set in the coefficient--exponent plane. A compact image is
stabilizable if all its exponents are at most $3\sqrt{3}/2$. If the Jacobian
of the parameter map has rank two at an interior parameter point and the
mapped exponent is at least $3\sqrt{3}/2$, no compact neighborhood of that
point is robustly stabilizable. A finite exponent image is stabilizable
whenever its largest exponent is below $4$.
\end{enumerate}

The paper is organized as follows. Section~\ref{sec:formulation} gives the
information pattern and stability definitions, and Section~\ref{sec:results}
states the main results. Sections~\ref{sec:positive} and~\ref{sec:negative}
prove the stabilization and impossibility results, respectively.
Section~\ref{sec:conclusion} concludes the paper. The appendices contain the
technical estimates, the proof for a finite set of exponents, and the proofs
of the two extension results.

\section{Problem formulation}\label{sec:formulation}

For system~\eqref{eq:system}, write
$p=(a,b)\in\R\times(0,\infty)$ and
$G_p(y)=G_{a,b}(y)\defeq a\phi_b(y)$.
The true parameter $p^\star=(a^\star,b^\star)$ is unknown. For a fixed known
$W>0$, the disturbance sequence satisfies
\begin{equation}\label{eq:disturbance}
 \sup_{t\ge1}|w_t|\le W.
\end{equation}
At time $t$, the controller observes $(y_0,\ldots,y_t)$ before choosing $u_t$.
Neither $p^\star$ nor $w_{t+1}$ is observed.

\begin{definition}\label{def:causal-feedback}
A feedback law is a sequence $\mu=(\mu_t)_{t\ge0}$ of maps
$\mu_t:\R^{t+1}\to\R$ such that
$u_t=\mu_t(y_0,\ldots,y_t)$.
\end{definition}

\begin{definition}\label{def:global-robust}
Let $\mathcal P\subset\R\times(0,\infty)$ be nonempty, and suppose
that $\mathcal P$ is known to the controller. A feedback law $\mu$
robustly stabilizes \eqref{eq:system} over $\mathcal P$ if
\[
 \sup_{t\ge0}|y_t|<\infty
\]
for every $p^\star\in\mathcal P$, every $y_0\in\R$, and every disturbance
sequence satisfying \eqref{eq:disturbance}.
The set $\mathcal P$ is robustly stabilizable if such a feedback law exists.
\end{definition}

The feedback law may depend on $\mathcal P$ and $W$. The finite bound in
Definition~\ref{def:global-robust} may depend on $p^\star$, $y_0$, and the
disturbance sequence. This is the boundedness criterion in
\cite[Definitions~2.1--2.2 and Remark~2.1]{LiXie2006} and
\cite[Definitions~2.1--2.2]{LiXieGuo2006}. The positive results below prove the
stronger estimate that, for every $Y\ge0$, the same feedback law satisfies
\begin{equation}\label{eq:uniform-state-bound}
 \sup_{\substack{p^\star\in\mathcal P,\ |y_0|\le Y\\
 (w_k)_{k\ge1}\text{ satisfies }\eqref{eq:disturbance}}}
 \ \sup_{t\ge0}|y_t|<\infty.
\end{equation}

Li and Xie \cite{LiXie2006} study the case in which $b=b_0$ is known and the
unknown coefficient $a$ ranges over a prescribed interval
$\mathcal A\subset\R$. Their uncertainty set is therefore
$\mathcal A\times\{b_0\}$. When $b$ is also unknown, we ask whether
stabilizability persists when both parameters vary in a neighborhood of a
nominal point. This leads to the following definition.

\begin{definition}\label{def:local-robust}
A point $p_0=(a_0,b_0)\in\R\times(0,\infty)$ is locally robustly
stabilizable if there exists a compact rectangle
$Q\subset\R\times(0,\infty)$ such that $p_0\in\operatorname{int}Q$ and $Q$ is
robustly stabilizable.
\end{definition}

The problem studied in this paper is to characterize the set of points
$p_0\in\R\times(0,\infty)$ that are locally robustly stabilizable.

\section{Main results}\label{sec:results}

The main result characterizes local robust stabilizability when both the
coefficient and exponent are unknown. A second result identifies the critical
exponent when the exponent is restricted to a known finite set. The remaining
results treat a known positive state-dependent multiplier and nonlinear
parameter maps.

\subsection{Exact critical exponent}

\begin{theorem}\label{thm:main}
Fix $W>0$ and consider \eqref{eq:system} under the disturbance bound
\eqref{eq:disturbance}.
\begin{enumerate}[label=\textup{(\roman*)}]
\item \emph{Stabilization.} Let $\mathcal P\subset\R\times(0,\infty)$ be
nonempty and compact, and put
$\overline b=\max\{b:(a,b)\in\mathcal P\}$. If
$\overline b\le3\sqrt{3}/2$, then there is a feedback law satisfying
\eqref{eq:uniform-state-bound} for every $Y\ge0$. In particular, $\mathcal P$
is robustly stabilizable.

\item \emph{Impossibility.} Let
$p_0=(a_0,b_0)\in\R\times(0,\infty)$ with $b_0>3\sqrt{3}/2$. Then $p_0$ is
not locally robustly stabilizable. More precisely, for every compact rectangle
$Q\subset\R\times(0,\infty)$ with $p_0\in\operatorname{int}Q$ and every
feedback law and every $R_0>0$, there exist $p^\star\in Q$,
$y_0\ge R_0$, a disturbance sequence $(w_t)_{t\ge1}$ satisfying
\eqref{eq:disturbance}, constants $c>0$ and $z>1$, and
$T_0\in\mathbb N$ such that
\[
 \log|y_t|\ge cz^t
\]
for every $t\ge T_0$. In particular,
$\sup_{t\ge0}|y_t|=\infty$.
\end{enumerate}
\end{theorem}

Theorem~\ref{thm:main}(i) includes compact parameter sets whose largest
exponent equals $3\sqrt{3}/2$. If a parameter point has exponent
$3\sqrt{3}/2$, however, every rectangle containing that point in its interior
also contains parameters with larger exponents. The exact local classification
is therefore as follows.

\begin{corollary}\label{cor:local-threshold}
A point $p_0=(a_0,b_0)\in\R\times(0,\infty)$ is locally robustly
stabilizable if and only if
\[
 b_0<\frac{3\sqrt{3}}{2}.
\]
\end{corollary}

\begin{proof}
If $b_0<3\sqrt{3}/2$, choose a compact rectangle $Q$ with
$p_0\in\operatorname{int}Q$ and largest exponent below $3\sqrt{3}/2$, and
apply Theorem~\ref{thm:main}(i). If $b_0>3\sqrt{3}/2$, apply
Theorem~\ref{thm:main}(ii).

Suppose $b_0=3\sqrt{3}/2$ and let $Q$ be any compact rectangle with
$p_0\in\operatorname{int}Q$. Choose
$p_1=(a_1,b_1)\in\operatorname{int}Q$ with $b_1>3\sqrt{3}/2$.
Theorem~\ref{thm:main}(ii) shows that $p_1$ is not locally robustly
stabilizable. If $Q$ were robustly stabilizable, it would witness the local
robust stabilizability of $p_1$, a contradiction.
\end{proof}

\begin{remark}\label{rem:bounded-random-disturbances}
The feedback law constructed in Section~\ref{sec:positive} is Borel measurable
by Lemma~\ref{lem:borel-controller}. Hence Theorem~\ref{thm:main}(i) applies
pathwise for every deterministic initial state whenever
$\sup_{t\ge1}|w_t|\le W$ almost surely. The resulting state trajectory is then
bounded almost surely, without any independence assumption.
\end{remark}

\subsection{Finite exponent uncertainty}

Corollary~\ref{cor:local-threshold} characterizes local robust stabilizability
when the exponent ranges over a neighborhood. We next determine the critical
exponent when the exponent is restricted to a known finite set.

\begin{theorem}\label{thm:finite-set}
Let $\mathcal A=[a_{\min},a_{\max}]\subset\R$ be a nondegenerate compact interval,
which may contain zero, and let $\mathcal B\subset(0,\infty)$ be a nonempty
finite set.
The family $\mathcal A\times\mathcal B$ is robustly stabilizable if and only if
$\max_{b\in\mathcal B}b<4$. When the maximum is below $4$, there is a
feedback law satisfying \eqref{eq:uniform-state-bound} with
$\mathcal P=\mathcal A\times\mathcal B$ for every $Y\ge0$.
\end{theorem}

For the converse, nondegeneracy guarantees that $\mathcal A$ contains a
nondegenerate compact subinterval disjoint from zero.

Theorem~\ref{thm:finite-set} includes the case of a single candidate exponent.
To isolate the effect of allowing the exponent to vary continuously, we next
compare a finite set with its nondegenerate interval hull.

\begin{corollary}
\label{cor:finite-continuous}
Let $\mathcal A\subset\R$ be a nondegenerate compact interval, and let
$\mathcal B\subset(0,\infty)$ be a nonempty finite set. Put
$b_{\min}=\min_{b\in\mathcal B}b$ and
$b_{\max}=\max_{b\in\mathcal B}b$. If $b_{\min}<b_{\max}$ and
$3\sqrt{3}/2<b_{\max}<4$, then $\mathcal A\times\mathcal B$ is
robustly stabilizable, whereas
$\mathcal A\times[b_{\min},b_{\max}]$ is not.
\end{corollary}

\begin{proof}
By Theorem~\ref{thm:finite-set}, the family $\mathcal A\times\mathcal B$ is
robustly stabilizable because $b_{\max}<4$. Since
$b_{\max}>3\sqrt{3}/2$, choose $a_0\in\operatorname{int}\mathcal A$ and
$b_0\in(b_{\min},b_{\max})$ such that $b_0>3\sqrt{3}/2$.
Then $(a_0,b_0)$ lies in the interior of
$\mathcal A\times[b_{\min},b_{\max}]$ and is not locally robustly
stabilizable by Corollary~\ref{cor:local-threshold}. Hence
$\mathcal A\times[b_{\min},b_{\max}]$
cannot be robustly stabilizable.
\end{proof}

The finite exponent set and its interval hull have the same minimum, maximum,
and diameter, yet only the finite family is stabilizable. Since the finite set
can lie in an arbitrarily short subinterval of $(3\sqrt{3}/2,4)$, the gap
persists for arbitrarily small positive diameter.

\subsection{Extensions of the model}

The first extension considers a known positive state-dependent multiplier
that is bounded above and away from zero.
For this family, robust and local robust stabilizability are defined as in
Definitions~\ref{def:global-robust} and~\ref{def:local-robust}, with the
dynamics below in place of \eqref{eq:system}.

\begin{corollary}\label{cor:shape-factor}
Let $h:\R\to(0,\infty)$ be known and continuous, and suppose that known
constants $\underline h$ and $\overline h$ satisfy
$0<\underline h\le h(y)\le\overline h<\infty$ for every $y\in\R$.
For the family
\[
 y_{t+1}=a h(y_t)\phi_b(y_t)+u_t+w_{t+1},
 \qquad a\in\R,\quad b>0,
\]
every nonempty compact parameter set
$\mathcal P\subset\R\times(0,\infty)$ whose largest exponent is at most
$3\sqrt{3}/2$ admits a feedback law satisfying
\eqref{eq:uniform-state-bound} for every $Y\ge0$. Moreover, a point
$p_0=(a_0,b_0)\in\R\times(0,\infty)$ is locally robustly stabilizable if and
only if $b_0<3\sqrt{3}/2$.
For $b_0>3\sqrt{3}/2$, the escape conclusion in
Theorem~\ref{thm:main}(ii) also holds for this family.
\end{corollary}

Another extension allows the coefficient and exponent to depend nonlinearly
on an underlying parameter. Let $q\in\mathbb N$, let
$\Xi\subset\R^q$ be nonempty and compact, and let $g_0:\R\to\R$,
$A:\Xi\to\R$, and $B:\Xi\to(0,\infty)$ be known, with $A$ and $B$
continuous. Consider
\begin{equation}\label{eq:general-parameter-map}
 y_{t+1}=g_0(y_t)+A(\xi)\phi_{B(\xi)}(y_t)+u_t+w_{t+1},
 \qquad \xi\in\Xi,
\end{equation}
and define $H(\xi)\defeq (A(\xi),B(\xi))$.
Robust stabilization over a compact set $K\subset\Xi$ is understood as in
Definition~\ref{def:global-robust}, with $\xi\in K$ in place of
$p\in\mathcal P$.

\begin{corollary}\label{cor:parameter-map}
For system~\eqref{eq:general-parameter-map}, the following statements hold.
\begin{enumerate}[label=\textup{(\roman*)}]
\item If $\sup_{\xi\in\Xi}B(\xi)\le3\sqrt{3}/2$, then the system is
robustly stabilizable over $\Xi$.

\item Let $\xi_0$ be an interior point of $\Xi$. Suppose that $A$ and $B$ are
continuously differentiable near $\xi_0$, that
$\operatorname{rank}DH(\xi_0)=2$, and that
$B(\xi_0)\ge3\sqrt{3}/2$.
Then no compact neighborhood $K\subset\Xi$ with
$\xi_0\in\operatorname{int}K$ is robustly stabilizable.

\item If $B(\Xi)$ is finite and $\sup_{\xi\in\Xi}B(\xi)<4$, then the system
is robustly stabilizable over $\Xi$.
\end{enumerate}
In each of parts~(i) and~(iii), there is a feedback law such that, for
every $Y\ge0$, the state is bounded uniformly over $\xi\in\Xi$, all
$|y_0|\le Y$, and all disturbance sequences satisfying \eqref{eq:disturbance}.
\end{corollary}

\begin{example}\label{ex:fluid-viscous-damper}
The force--velocity relation
\[
 F_d(v)=c_d|v|^\nu\operatorname{sgn}(v),
 \qquad c_d>0,\quad \nu>0,
\]
is used to model nonlinear fluid viscous dampers. The exponent describes the
nonlinear hydraulic behavior of the device
\cite{DeDomenicoRicciardi2018,MoslehiTabarDeDomenico2020}. Consider a body of
known mass $m>0$ driven through such a damper,
\[
 m\dot v=-F_d(v)+\bar u+\bar w.
\]
Here $\bar u$ is the control force and $\bar w$ is the disturbance force.
Fix a sampling period $\Delta>0$ and set
$u_t=(\Delta/m)\bar u_t$ and $w_{t+1}=(\Delta/m)\bar w_t$.
Forward Euler discretization yields the discrete-time approximation
\[
 v_{t+1}
 =v_t-\frac{\Delta c_d}{m}\phi_\nu(v_t)+u_t+w_{t+1}.
\]
This is \eqref{eq:general-parameter-map} with $g_0(v)=v$,
$\xi=(c_d,\nu)$, and
$H(\xi)=(-\Delta c_d/m,\nu)$.
On any compact parameter rectangle $\Xi\subset(0,\infty)^2$, the map
satisfies the continuity assumptions of Corollary~\ref{cor:parameter-map}.
Moreover,
\[
 DH(\xi)
 =\begin{pmatrix}
   -\Delta/m & 0\\
   0 & 1
  \end{pmatrix}.
\]
Corollary~\ref{cor:parameter-map}(i) gives
robust stabilization under \eqref{eq:disturbance} on every compact rectangle
$\Xi$ with $\sup_{(c_d,\nu)\in\Xi}\nu\le3\sqrt{3}/2$.
Since $\operatorname{rank}DH(\xi)=2$ for every $\xi\in\Xi$,
Corollary~\ref{cor:parameter-map}(ii) also shows that if
$(c_d,\nu)\in\operatorname{int}\Xi$ and $\nu\ge3\sqrt{3}/2$, no compact
neighborhood $K\subset\Xi$ of $(c_d,\nu)$ is robustly stabilizable.
\end{example}

\section{Proof of Theorem~\ref{thm:main}(i)}\label{sec:positive}

The feedback law applies zero input until the state first leaves
$[-2W,2W]$. If the state never leaves this interval, the trajectory is
bounded. Otherwise, the transition producing the first exit implies that all
parameters consistent with the observations have coefficients of the same
nonzero sign and magnitudes bounded away from zero. It is therefore enough to
prove a state bound that is uniform over compact parameter sets whose
coefficient magnitudes are bounded away from zero.

\subsection{Feedback law and the first exit}

Fix the known compact parameter set $\mathcal P$. At time $t$, define the
parameters consistent with all observations available by that time as
\begin{equation}\label{eq:feasible-set}
 \mathcal P_t
 \defeq \left\{p\in\mathcal P:
 \left|y_{k+1}-u_k-G_p(y_k)\right|\le W,
 \ 0\le k<t\right\}.
\end{equation}
On a trajectory generated by $p^\star\in\mathcal P$ under
\eqref{eq:disturbance},
$y_{k+1}-u_k-G_{p^\star}(y_k)=w_{k+1}$, so
$p^\star\in\mathcal P_t$ for every $t$. Thus $\mathcal P_t$ is nonempty and
compact, and
$\mathcal P_{t+1}\subseteq\mathcal P_t$. Define the parameters in
$\mathcal P_t$ having the largest exponent by
\begin{equation}\label{eq:maximal-exponent-fiber}
 \mathcal F_t
 \defeq \operatorname*{arg\,max}_{(a,b)\in\mathcal P_t} b.
\end{equation}
For every nonempty compact parameter set $\mathcal S$, define
$m_t(\mathcal S)$ as the arithmetic mean of the smallest and largest values of
$G_p(y_t)$ over $p\in\mathcal S$:
\[
 m_t(\mathcal S)\defeq \frac12\left(
 \min_{p\in\mathcal S}G_p(y_t)+\max_{p\in\mathcal S}G_p(y_t)
 \right).
\]
Define the first-exit time
$t_{\rm a}\defeq \inf\{s\in\mathbb N:s\ge1,\ |y_s|>2W\}$,
where the infimum of the empty set is infinity. Whether $t\ge t_{\rm a}$ is
determined by the states observed by time $t$. On histories for which
$\mathcal P_t$ is nonempty, the controller is
\begin{equation}\label{eq:fiber-midpoint-controller}
 u_t=
 \begin{cases}
  0,&t<t_{\rm a},\\
  -m_t(\mathcal F_t),&t\ge t_{\rm a}.
 \end{cases}.
\end{equation}
On histories for which $\mathcal P_t$ is empty, set $u_t=0$.
At time $t$, every quantity in
\eqref{eq:fiber-midpoint-controller} is determined by
$\mathcal P$, $W$, and $(y_0,\ldots,y_t)$.

The zero-input transition at $t_{\rm a}$ gives the following reduction.

\begin{lemma}
\label{lem:first-exit-reduction}
Let $\mathcal P\subset\R\times(0,\infty)$ be nonempty and compact, and set
\[
 \overline a=\max_{(a,b)\in\mathcal P}|a|,
 \qquad
 b_-=\min_{(a,b)\in\mathcal P}b,
 \qquad
 \overline b=\max_{(a,b)\in\mathcal P}b.
\]
Fix $Y\ge0$, $p^\star\in\mathcal P$, $|y_0|\le Y$, and a disturbance sequence
satisfying \eqref{eq:disturbance}. Initialize \eqref{eq:feasible-set} with
$\mathcal P$, use zero input for $t<t_{\rm a}$, and suppose that
$t_{\rm a}<\infty$. With
\[
 S=\max\{1,Y,2W\},
 \qquad
 y_{\rm a}=y_{t_{\rm a}},
\]
the set $\mathcal P_{t_{\rm a}}$ is nonempty and compact, all its coefficients
have the same nonzero sign, and
\begin{equation}\label{eq:first-exit-reduction}
 \mathcal P_{t_{\rm a}}\subset
 \left\{(a,b):
 \frac{W}{S^{\overline b}}<|a|\le\overline a,
 \quad b_-\le b\le\overline b
 \right\},
 \qquad
 |y_{\rm a}|\le\overline aS^{\overline b}+W.
\end{equation}
\end{lemma}

\begin{proof}
The true parameter belongs to $\mathcal P_{t_{\rm a}}$, which is compact
because it is obtained from $\mathcal P$ by finitely many closed feasibility
constraints. Since $u_{t_{\rm a}-1}=0$, every
$(a,b)\in\mathcal P_{t_{\rm a}}$
satisfies
\[
 a\phi_b(y_{t_{\rm a}-1})
 \in[y_{\rm a}-W,y_{\rm a}+W].
\]
The interval on the right does not contain zero because $|y_{\rm a}|>2W$.
Hence $y_{t_{\rm a}-1}\ne0$, all feasible coefficients have the same nonzero
sign, and $|a|\,|y_{t_{\rm a}-1}|^b>W$.
The definition of $t_{\rm a}$ and the bound on the initial state give
$|y_{t_{\rm a}-1}|\le S$. Thus $|a|>W/S^{\overline b}$, while
$|a|\le\overline a$ and $b_-\le b\le\overline b$ follow from
$\mathcal P_{t_{\rm a}}\subset\mathcal P$. For the
true parameter,
$|y_{\rm a}|\le |a^\star||y_{t_{\rm a}-1}|^{b^\star}+W
\le\overline aS^{\overline b}+W$, which completes the proof of
\eqref{eq:first-exit-reduction}.
\end{proof}

\subsection{Uniform state bound}

Lemma~\ref{lem:first-exit-reduction} reduces the proof to compact parameter
sets whose coefficient magnitudes are bounded away from zero. Since the
resulting set $\mathcal P_{t_{\rm a}}$ depends on the observed history, we
establish a common state bound for all compact sets contained in fixed
coefficient and exponent ranges.

\begin{proposition}
\label{prop:away-zero-core}
Fix $W>0$, $0<\underline a\le\overline a<\infty$, and
$0<b_-\le\overline b\le3\sqrt{3}/2$.
For every $Y\ge0$, there is a constant $C<\infty$, depending only on
$Y,\underline a,\overline a,b_-,\overline b$, and $W$, with the following
property. For every nonempty compact set
\[
 \mathcal P\subset
 \{(a,b):\underline a\le|a|\le\overline a,\quad
 b_-\le b\le\overline b\},
\]
the rule $u_t=-m_t(\mathcal F_t)$, with $\mathcal P_t$ initialized by
$\mathcal P$, satisfies
\[
 \sup_{\substack{p^\star\in\mathcal P,\ |y_0|\le Y,\ t\ge0\\
                  (w_k)_{k\ge1}:\ \sup_{k\ge1}|w_k|\le W}}
 |y_t|\le C.
\]
\end{proposition}

\begin{proof}
Let $p^\star=(a^\star,b^\star)\in\mathcal P$, $|y_0|\le Y$, and let
$(w_t)_{t\ge1}$ satisfy \eqref{eq:disturbance}. With $\mathcal P_t$ and
$\mathcal F_t$ defined by \eqref{eq:feasible-set} and
\eqref{eq:maximal-exponent-fiber}, put
\[
 D_t\defeq
 \max_{p\in\mathcal F_t}|G_p(y_t)-G_{p^\star}(y_t)|.
\]
The maximum exists because $\mathcal F_t$ is nonempty and compact.
Every value $G_p(y_t)$ with $p\in\mathcal F_t$ lies within $D_t$ of
$G_{p^\star}(y_t)$. Hence
$|m_t(\mathcal F_t)-G_{p^\star}(y_t)|\le D_t$. Since
$u_t=-m_t(\mathcal F_t)$,
\eqref{eq:system} and \eqref{eq:disturbance} give
\begin{equation}\label{eq:current-fiber-state}
 |y_{t+1}|\le W+D_t.
\end{equation}

Set $\tau_0=0$ and $L_0=\log\max\{1,|y_0|\}$. For each $n\ge0$ with
$\tau_n<\infty$, define
\begin{equation}\label{eq:record-times}
 \tau_{n+1}=\inf\{t>\tau_n:
 \max\{1,|y_t|\}>2e^{L_n}\}.
\end{equation}
If $\tau_{n+1}<\infty$, set
$L_{n+1}=\log\max\{1,|y_{\tau_{n+1}}|\}$. The definition gives
\begin{align}
 \max\{1,|y_t|\}
 &\le2e^{L_n},\qquad \tau_n\le t<\tau_{n+1},
 \label{eq:between-records}\\
 L_{n+1}
 &>L_n+\log 2,\qquad \tau_{n+1}<\infty.
 \label{eq:record-increment}
\end{align}
If $\tau_{n+1}=\infty$, then \eqref{eq:between-records} gives
$\sup_{t\ge\tau_n}|y_t|\le2e^{L_n}$. Fix $n$ with
$\tau_{n+1}<\infty$.

For $\tau_n\le t<\tau_{n+1}$, both $|G_{a,b}(y_t)|$, uniformly over
$(a,b)\in\mathcal P$, and $|m_t(\mathcal F_t)|$ are at most
$\overline a2^{\overline b}e^{\overline bL_n}$. Use the following constant in
all three record estimates:
\begin{equation}\label{eq:record-C0}
 C_0=\log\left(
 2+W+2^{\overline b+1}(\overline a+W)
 +\frac{16eW\,2^{\overline b}}{\log2}
 \right).
\end{equation}
Applying the state equation at $t=\tau_{n+1}-1$ gives
\[
 e^{L_{n+1}}
 \le 1+W+2^{\overline b+1}\overline a
             e^{\overline bL_n}
 \le e^{C_0+\overline bL_n}.
\]
Therefore
\begin{equation}\label{eq:top-coarse-record}
 L_{n+1}\le\overline bL_n+C_0.
\end{equation}
In particular,
$L_n\le\overline b^2L_{n-2}+(\overline b+1)C_0$ for $n\ge2$.

Let $R_*=\max\{1,2\overline a+W\}$. When $|y_t|\le1$, the plant drift and
$m_t(\mathcal F_t)$ both have magnitude at most $\overline a$. Hence
\begin{equation}\label{eq:top-small-state}
 \max\{1,|y_{t+1}|\}\le R_*
 \qquad\text{whenever } |y_t|\le1.
\end{equation}

We first derive the record bound that uses two previous record states. Since
$b_->0$, choose
$\Lambda_0\ge1$ so that, for every $x\ge\Lambda_0$,
\begin{equation}\label{eq:top-record-threshold}
 \underline a e^{b_-x}\ge4W,
 \quad 4e^x>R_*,\quad
 \frac{16W}{\underline a\log2}
 \bigl[1+\overline b^2x+(\overline b+1)C_0\bigr]e^{-b_-x}\le1.
\end{equation}
The threshold $\Lambda_0$ depends only on
$\underline a,\overline a,b_-,\overline b$, and $W$. Let $n\ge2$ satisfy
$\tau_{n+1}<\infty$ and $L_{n-2}\ge\Lambda_0$, and put
$t=\tau_{n+1}-1$. Since
$\tau_{n-2}<\tau_{n-1}<t$, the observations generated from
$y_{\tau_{n-2}}$ and $y_{\tau_{n-1}}$ occur among the constraints defining
$\mathcal P_t$. For $i=n-2,n-1$ and $p'\in\mathcal F_t$, feasibility and
the first inequality in \eqref{eq:top-record-threshold} give
\[
 |G_{p'}(y_{\tau_i})-G_{p^\star}(y_{\tau_i})|
 \le2W,
 \qquad
 |G_{p^\star}(y_{\tau_i})|
 \ge\underline a e^{b_-L_{n-2}}\ge4W.
\]
Hence every $p'=(a',b')\in\mathcal F_t$ satisfies $a'a^\star>0$.
Eq.~\eqref{eq:record-increment} and the second inequality in
\eqref{eq:top-record-threshold} give
$e^{L_{n+1}}>8e^{L_{n-2}}>R_*$, so
\eqref{eq:top-small-state} implies $|y_t|>1$.

Fix $p'=(a',b')\in\mathcal F_t$ and apply
Lemma~\ref{lem:two-scale} to the observed states
$y_{\tau_{n-2}}$, $y_{\tau_{n-1}}$, and the current state $y_t$.
The relations
$0<L_{n-2}<L_{n-1}<L_n$, $0<\log|y_t|\le L_n+\log2$, and
$L_{n-1}-L_{n-2}>\log2$ give
\[
 \frac{\bigl|L_{n-1}-\log|y_t|\bigr|
       +\bigl|\log|y_t|-L_{n-2}\bigr|}
 {L_{n-1}-L_{n-2}}
 \le \frac{2(L_n+\log2)}{\log2}
 \le \frac{4(1+L_n)}{\log2}.
\]
Consequently, \eqref{eq:two-state-log-ratio} gives
\[
 \left|\log\frac{|G_{p'}(y_t)|}{|G_{p^\star}(y_t)|}\right|
 \le \frac{16W}{|a^\star|\log2}(1+L_n)
 e^{-b^\star L_{n-2}},
 \qquad p'\in\mathcal F_t.
\]

Using \eqref{eq:top-coarse-record},
$|a^\star|\ge\underline a$, $b_-\le b^\star\le\overline b$, and the third
inequality in \eqref{eq:top-record-threshold}, the right-hand side satisfies
\[
 \frac{16W}{\underline a\log2}
 \bigl[1+\overline b^2L_{n-2}+(\overline b+1)C_0\bigr]
 e^{-b_-L_{n-2}}\le1.
\]
By \eqref{eq:between-records}, $|y_t|\le2e^{L_n}$. Hence
$|G_{p^\star}(y_t)|=|a^\star||y_t|^{b^\star}
\le |a^\star|2^{\overline b}e^{b^\star L_n}$.
Eq.~\eqref{eq:two-state-drift-error}, followed by
maximization over
$p'\in\mathcal F_t$, now gives
\[
 D_t\le \frac{16eW\,2^{\overline b}}{\log2}(1+L_n)
 \exp\{b^\star(L_n-L_{n-2})\}.
\]

Substitution in \eqref{eq:current-fiber-state} gives
\[
 e^{L_{n+1}}
 \le 1+W+D_t
 \le e^{C_0}(1+L_n)
       \exp\{b^\star(L_n-L_{n-2})\}.
\]
Taking logarithms gives
\begin{equation}\label{eq:top-two-scale-record}
 L_{n+1}
 \le b^\star(L_n-L_{n-2})+\log(1+L_n)+C_0,
\end{equation}
whenever $n\ge2$, $\tau_{n+1}<\infty$, and $L_{n-2}\ge\Lambda_0$.

We next derive the record bound used when $b^\star$ is close to
$\overline b$. Set
\begin{equation}\label{eq:top-upper-face-kappa}
 \kappa=\min\left\{1,\frac{1}{e\overline a\,2^{\overline b}}\right\}.
\end{equation}
Suppose that $n\ge1$, $\tau_{n+1}<\infty$,
$L_{n-1}\ge\Lambda_0$, and
\[
 \overline b-b^\star
 \le\kappa\frac{e^{-\overline bL_n}}{1+L_n}.
\]
At time
$t=\tau_{n+1}-1$, the transition from $y_{\tau_{n-1}}$ has been observed.
For every $p'\in\mathcal F_t$, feasibility and the first inequality in
\eqref{eq:top-record-threshold} give
\[
 |G_{p'}(y_{\tau_{n-1}})-G_{p^\star}(y_{\tau_{n-1}})|
 \le2W<|G_{p^\star}(y_{\tau_{n-1}})|,
\]
so $p'$ has the coefficient sign of $p^\star$. Because
$p^\star\in\mathcal P_t$ and $\mathcal F_t$ maximizes the exponent over
$\mathcal P_t$, every $p'=(a',b')\in\mathcal F_t$ also satisfies
$b^\star\le b'\le\overline b$. Moreover,
Eq.~\eqref{eq:record-increment} and the second inequality in
\eqref{eq:top-record-threshold} give
$e^{L_{n+1}}>4e^{L_{n-1}}>R_*$. Hence \eqref{eq:top-small-state} implies
$|y_t|>1$. Thus
\[
 1<|y_{\tau_{n-1}}|\le e^{L_n},
 \qquad
 1<|y_t|\le2e^{L_n}.
\]
Since $L_n+\log2\le1+L_n$ and
$\overline b-b^\star\le\kappa e^{-\overline bL_n}/(1+L_n)$,
\eqref{eq:top-upper-face-kappa} gives
\[
 (\overline b-b^\star)(L_n+\log2)
 \le \kappa e^{-\overline bL_n}\le\kappa\le1 
\]
and
\[
 \overline a e^{\overline b(L_n+\log2)}
 (\overline b-b^\star)(L_n+\log2)
 e^{(\overline b-b^\star)(L_n+\log2)}
 \le e\overline a\,2^{\overline b}\kappa\le1.
\]
Lemma~\ref{lem:one-anchor-upper-face}, applied with
$\Lambda=L_n+\log2$, therefore gives
\[
 D_t\le 2^{\overline b+1}W
 \exp\{\overline b(L_n-L_{n-1})\}+1.
\]
Consequently,
\[
 e^{L_{n+1}}
 \le1+W+D_t
 \le e^{C_0+\overline b(L_n-L_{n-1})}.
\]
Thus
\begin{equation}\label{eq:top-one-scale-record}
 L_{n+1}\le\overline b(L_n-L_{n-1})+C_0.
\end{equation}

For every choice of $\mathcal P$, $p^\star$, $y_0$, and
$(w_t)_{t\ge1}$ allowed by the proposition, and every $N$ such that
$\tau_N<\infty$, the sequence $(L_0,\ldots,L_N)$ satisfies the hypotheses of
Lemma~\ref{lem:record-closure}. Take $b=b^\star$,
$L_{\max}=\log(\max\{1,Y\})$, $C_0$ from \eqref{eq:record-C0},
$\Lambda_0$ from \eqref{eq:top-record-threshold}, and $\kappa$ from
\eqref{eq:top-upper-face-kappa}. The parameter bounds in the proposition and
the definition of $L_0$ give $0<b^\star\le\overline b$ and
$0\le L_0\le L_{\max}$. The required recurrence bounds are
\eqref{eq:record-increment}, \eqref{eq:top-coarse-record},
\eqref{eq:top-two-scale-record}, and \eqref{eq:top-one-scale-record}.
Let
$M=M(\log(\max\{1,Y\}),C_0,\Lambda_0,\overline b,\kappa)$ be the bound
given by Lemma~\ref{lem:record-closure}. Then $L_n\le M$ at every finite
record, and \eqref{eq:between-records} gives
$\max\{1,|y_t|\}\le2e^M$ for every $t\ge0$. This is the required uniform
state bound.
\end{proof}

\begin{proof}[Proof of Theorem~\ref{thm:main}(i)]
Fix $Y\ge0$, $p^\star\in\mathcal P$, $|y_0|\le Y$, and
a disturbance sequence satisfying \eqref{eq:disturbance}.
If $t_{\rm a}=\infty$, then $|y_t|\le2W$ for every $t\ge1$.

Suppose that $t_{\rm a}<\infty$ and use the notation of
Lemma~\ref{lem:first-exit-reduction}. For every integer $s\ge0$, reinitializing
the feasible set at time $t_{\rm a}$ with $\mathcal P_{t_{\rm a}}$ gives
\[
 \left\{p\in\mathcal P_{t_{\rm a}}:
 |y_{k+1}-u_k-G_p(y_k)|\le W,
 \ t_{\rm a}\le k<t_{\rm a}+s\right\}
 =\mathcal P_{t_{\rm a}+s}.
\]
The restarted controller therefore agrees with
\eqref{eq:fiber-midpoint-controller} at every subsequent time. Apply
Proposition~\ref{prop:away-zero-core} from time $t_{\rm a}$ with coefficient
bounds $W/S^{\overline b}$ and $\overline a$, exponent bounds $b_-$ and
$\overline b$, and radius $\overline aS^{\overline b}+W$ for the state at
time $t_{\rm a}$. These quantities depend only on
$\mathcal P$, $W$, and $Y$. The proposition therefore bounds the state
uniformly for $t\ge t_{\rm a}$. Together with $|y_0|\le Y$ and
$|y_t|\le2W$ for
$1\le t<t_{\rm a}$, this proves the theorem.
\end{proof}

\section{Proof of Theorem~\ref{thm:main}(ii)}\label{sec:negative}

The following proposition constructs the escaping trajectory used to prove
Theorem~\ref{thm:main}(ii).

\begin{proposition}\label{prop:escape}
Let $\mathcal A\subset\R$ be a nondegenerate compact interval with
$0\notin\mathcal A$, and let $1<b_-<b_+$. Set
\[
 Q=\mathcal A\times[b_-,b_+],
\]
and suppose that there is a $z>1$ such that
\begin{equation}\label{eq:strict-escape}
 b_- -\frac{b_+}{z^2}>z.
\end{equation}
Then there exists $L_\star\ge1$, depending only on $Q,W$, and $z$, such
that, for every feedback law and every $L_0\ge L_\star$, there exist
$p^\star\in Q$ and a disturbance sequence satisfying
\eqref{eq:disturbance} for which $y_0=e^{L_0}$ and
\[
 \log|y_t|\ge L_0z^t,
 \qquad t\ge0.
\]
In particular, no feedback law robustly stabilizes
\eqref{eq:system} over $Q$.
\end{proposition}

\begin{proof}
Fix a feedback law. Let $\sigma\in\{-1,1\}$ be the common sign of the
coefficients in $\mathcal A$, and define
\[
 T_\sigma(a,b)=(\log|a|,b),
 \qquad
 T_\sigma^{-1}(\alpha,b)=(\sigma e^\alpha,b).
\]
For $\eta=(\alpha,b)$ and $L\in\R$, put $\ell_\eta(L)=\alpha+bL$ and
$G_\eta=G_{T_\sigma^{-1}(\eta)}$. For $y\ne0$ and $L=\log|y|$,
\begin{equation}\label{eq:log-coordinate-drift}
 G_\eta(y)
 =\sigma\operatorname{sgn}(y)e^{\ell_\eta(L)}.
\end{equation}
The image of $Q$ under $T_\sigma$ is
\[
 \widetilde Q\defeq T_\sigma(Q)
 =[\alpha_-,\alpha_+]\times[b_-,b_+],
 \qquad
 \alpha_-=\min_{a\in\mathcal A}\log|a|,
 \qquad
 \alpha_+=\max_{a\in\mathcal A}\log|a|.
\]
We construct the trajectory for every sufficiently large $L_0$. Every lower
bound on $L_0$ imposed below depends only on $Q,W$, and $z$.

\medskip
\noindent\textbf{Step 1: Initialization.}
Put
\[
 \alpha_{\rm m}=\frac{\alpha_-+\alpha_+}{2},
 \qquad
 b_{\rm m}=\frac{b_-+b_+}{2},
 \qquad
 d=\frac{\alpha_+-\alpha_-}{8}.
\]
Let $y_0=e^{L_0}$. After the feedback law selects $u_0$, choose
\[
 \eta_0\in
  \operatorname*{arg\,max}_{\eta\in
  \{(\alpha_{\rm m}-d,b_{\rm m}),
    (\alpha_{\rm m}+d,b_{\rm m})\}}
 |u_0+G_\eta(y_0)|,
\]
and set $y_1=u_0+G_{\eta_0}(y_0)$. Since $b_{\rm m}>b_->z$, for all
sufficiently large $L_0$, the inequality
$\max\{|u+x|,|u+y|\}\ge|x-y|/2$ gives
\[
 \begin{aligned}
 |y_1|
 \ge \frac12
 \left|G_{(\alpha_{\rm m}+d,b_{\rm m})}(y_0)
       -G_{(\alpha_{\rm m}-d,b_{\rm m})}(y_0)\right|=\frac{e^d-e^{-d}}{2}e^{\alpha_{\rm m}+b_{\rm m}L_0}
 \ge e^{zL_0}.
\end{aligned}
\]
Thus $L_1=\log|y_1|\ge zL_0$.

Set $v_1=(1,-L_0^{-1})$. The identity $\ell_{v_1}(L_0)=0$ shows that
$G_{\eta_0}(y_0)=G_{\eta_0+dv_1}(y_0)$. After $u_1$
is selected, choose
\[
 \eta_1\in
 \operatorname*{arg\,max}_{\eta\in\{\eta_0,\eta_0+dv_1\}}
 |u_1+G_\eta(y_1)|,
\]
and set $y_2=u_1+G_{\eta_1}(y_1)$. The two values of $\ell_\eta(L_1)$
differ by $d(L_1/L_0-1)\ge d(z-1)$. Their first coordinates lie in
$[\alpha_{\rm m}-d,\alpha_{\rm m}+2d]$, and their second coordinates lie in
$[b_{\rm m}-d/L_0,b_{\rm m}]$. Thus both candidates belong to
$\widetilde Q$ for all sufficiently large $L_0$. For such $L_0$,
$b_->z$ and $L_1\ge zL_0$ give
\[
 \begin{aligned}
 |y_2|
 \ge\frac12
   \left|G_{\eta_0+dv_1}(y_1)-G_{\eta_0}(y_1)\right|\ge\frac12 e^{\alpha_-+b_-L_1}
       \bigl(e^{d(z-1)}-1\bigr)
 \ge e^{zL_1}.
 \end{aligned}
\]
Therefore $L_2=\log|y_2|\ge zL_1$.
Writing $\eta_1=(\alpha_1,b_1)$, the definitions of $\eta_0$ and $\eta_1$ give
\begin{equation}\label{eq:escape-initial-parameter-bounds}
 \alpha_{\rm m}-d\le\alpha_1\le\alpha_{\rm m}+2d,
 \qquad
 b_{\rm m}-\frac d{L_0}\le b_1\le b_{\rm m}.
\end{equation}

\medskip
\noindent\textbf{Step 2: Recursive construction.}
For each $k\ge0$ for which $L_k$ has been constructed, define
\[
 s_k=\frac{W(z-1)}{8z}e^{-\alpha_+-b_+L_k}.
\]
Eq.~\eqref{eq:escape-initial-parameter-bounds} and the definition of
$s_0$ show that the following two conditions hold for all sufficiently large
$L_0$:
\begin{equation}\label{eq:escape-large-initial-level}
 e^{-b_+(z-1)L_0}\le\frac12,
 \quad
 \left[\alpha_1,\alpha_1+\frac{2z}{z-1}s_0\right]
 \times
 \left[b_1-\frac{2s_0}{(z-1)L_0},b_1\right]
 \subset\widetilde Q.
\end{equation}
The strict inequality \eqref{eq:strict-escape} also allows $L_0$ to be chosen
so that
\begin{equation}\label{eq:escape-growth-margin}
 \log\!\left(\frac{W(z-1)^2}{16z}\right)
 +\alpha_--\alpha_++
 \left(b_--\frac{b_+}{z^2}-z\right)L_0\ge0.
\end{equation}

For $n\ge2$, once $L_n$ has been constructed, define
\begin{equation}\label{eq:escape-update-direction}
 v_n=\frac{(L_{n-1},-1)}{L_{n-1}-L_{n-2}}.
\end{equation}
After the feedback law selects $u_n$, choose
\begin{equation}\label{eq:escape-parameter-update}
 \eta_n\in
 \operatorname*{arg\,max}_{\eta\in
 \{\eta_{n-1},\eta_{n-1}+s_{n-2}v_n\}}
 |u_n+G_\eta(y_n)|,
\end{equation}
and set $y_{n+1}=u_n+G_{\eta_n}(y_n)$.

The initialization gives $L_1\ge zL_0$ and $L_2\ge zL_1$. We now show
inductively that the recursive choice remains in $\widetilde Q$ and satisfies
$L_{n+1}\ge zL_n$. Fix $n\ge2$ and suppose that $y_0,\ldots,y_n$ and
$\eta_0,\ldots,\eta_{n-1}$ have been constructed and that
$L_j\ge zL_{j-1}$ for $1\le j\le n$. Then
\begin{equation}\label{eq:escape-update-properties}
 \ell_{v_n}(L_{n-1})=0,
 \qquad
 |\ell_{v_n}(L_n)|
 =\frac{L_n-L_{n-1}}{L_{n-1}-L_{n-2}}
 \ge z-1,
\end{equation}
and
\begin{equation}\label{eq:escape-update-coordinate-bounds}
 \frac{L_{n-1}}{L_{n-1}-L_{n-2}}
 \le\frac z{z-1},
 \qquad
 \frac1{L_{n-1}-L_{n-2}}
 \le\frac1{(z-1)L_{n-2}}.
\end{equation}
For $0\le k\le n-2$, the induction hypothesis
$L_{k+1}\ge zL_k$ and the first inequality in
\eqref{eq:escape-large-initial-level} give
\[
 \frac{s_{k+1}}{s_k}
 =e^{-b_+(L_{k+1}-L_k)}
 \le e^{-b_+(z-1)L_0}\le\frac12.
\]
For $2\le j<n$, \eqref{eq:escape-parameter-update} gives
$\eta_j-\eta_{j-1}\in\{0,s_{j-2}v_j\}$.
For either candidate $\eta=(\alpha,b)$ in
\eqref{eq:escape-parameter-update},
\begin{equation}\label{eq:escape-candidate-displacement}
 0\le \alpha-\alpha_1
 \le\frac z{z-1}\sum_{k=0}^{n-2}s_k
 \le\frac{2z}{z-1}s_0,
 \quad
 0\le b_1-b
 \le\sum_{k=0}^{n-2}\frac{s_k}{(z-1)L_k}
 \le\frac{2s_0}{(z-1)L_0}.
\end{equation}
By \eqref{eq:escape-large-initial-level}, both candidates in
\eqref{eq:escape-parameter-update} belong to $\widetilde Q$.

Using $e^x-1\ge x$ for $x\ge0$, the choice in
\eqref{eq:escape-parameter-update}, together with
\eqref{eq:log-coordinate-drift} and
\eqref{eq:escape-update-properties}, yields
\[
 \begin{aligned}
 |y_{n+1}|
 &\ge\frac12
  \left|G_{\eta_{n-1}+s_{n-2}v_n}(y_n)
        -G_{\eta_{n-1}}(y_n)\right|\ge\frac12e^{\alpha_-+b_-L_n}
       s_{n-2}|\ell_{v_n}(L_n)|\\
 &\ge\frac{z-1}{2}s_{n-2}e^{\alpha_-+b_-L_n}=\frac{W(z-1)^2}{16z}
  e^{\alpha_--\alpha_++b_-L_n-b_+L_{n-2}}.
 \end{aligned}
\]
In particular, $y_{n+1}\ne0$, so $L_{n+1}=\log|y_{n+1}|$ is defined.
Since $L_{n-2}\le L_n/z^2$ and $L_n\ge L_0$,
\eqref{eq:escape-growth-margin} gives
\[
 \begin{aligned}
 L_{n+1}
 &\ge\log\!\left(\frac{W(z-1)^2}{16z}\right)
     +\alpha_--\alpha_++b_-L_n-b_+L_{n-2}\\
 &\ge\log\!\left(\frac{W(z-1)^2}{16z}\right)
     +\alpha_--\alpha_++
       \left(b_--\frac{b_+}{z^2}\right)L_n\ge zL_n.
 \end{aligned}
\]
This proves the induction step, so the construction continues for every
$n\ge2$.

\medskip
\noindent\textbf{Step 3: A fixed parameter and disturbance sequence.}
Using the $1$-norm on $\R^2$, the coordinate bounds in
\eqref{eq:escape-update-coordinate-bounds} and
$s_{n+1}\le s_n/2$ for $n\ge0$ give
\[
 \sum_{n=2}^\infty\|\eta_n-\eta_{n-1}\|_1
 \le\left(\frac z{z-1}+\frac1{(z-1)L_0}\right)
       \sum_{n=2}^\infty s_{n-2}<\infty.
\]
Including the initial increment gives
$\sum_{n=1}^\infty\|\eta_n-\eta_{n-1}\|_1<\infty$, so
$(\eta_n)_{n\ge0}$ is a Cauchy sequence. Write
\[
 \eta^\star
 =\lim_{n\to\infty}\eta_n\in\widetilde Q,
 \qquad
 p^\star=T_\sigma^{-1}(\eta^\star)\in Q.
\]
The initialization gives $\ell_{\eta_1-\eta_0}(L_0)=0$. For every $k\ge1$,
\eqref{eq:escape-parameter-update} and \eqref{eq:escape-update-properties} give
\[
 \ell_{\eta_{k+1}-\eta_k}(L_k)=0.
\]
For all $k\ge0$ and $n\ge k+2$,
\[
 |\ell_{v_n}(L_k)|
 =\frac{L_{n-1}-L_k}{L_{n-1}-L_{n-2}}
 \le\frac z{z-1}.
\]
Hence, for every $k\ge0$,
\begin{align*}
 |\ell_{\eta^\star-\eta_k}(L_k)|
 \le\frac z{z-1}\sum_{j=k}^\infty s_j
 \le\frac{2z}{z-1}s_k=\frac W4e^{-\alpha_+-b_+L_k}.
\end{align*}
Since both $\eta_k$ and $\eta^\star$ belong to $\widetilde Q$, the mean value
theorem and \eqref{eq:log-coordinate-drift} imply
\[
 \begin{aligned}
 |G_{\eta_k}(y_k)-G_{\eta^\star}(y_k)|
 \le e^{\alpha_++b_+L_k}
       |\ell_{\eta_k-\eta^\star}(L_k)|\le\frac W4.
 \end{aligned}
\]
Define the fixed disturbance sequence by
\[
 w_{k+1}=G_{\eta_k}(y_k)-G_{\eta^\star}(y_k),
 \qquad k\ge0.
\]
Then $|w_{k+1}|\le W/4\le W$, and the definition of $w_{k+1}$ shows that
\eqref{eq:system} holds with the fixed parameter $p^\star$ for every $k\ge0$.
The induction also gives
\[
 \log|y_k|=L_k\ge L_0z^k,
 \qquad k\ge0.
\]
The lower bounds on $L_0$ required in Steps 1 and 2 are satisfied once
$L_0\ge L_\star$ for some $L_\star\ge1$ depending only on $Q,W$, and $z$.
\end{proof}

\begin{proof}[Proof of Theorem~\ref{thm:main}(ii)]
Fix a compact rectangle $Q_0$ with
$p_0=(a_0,b_0)\in\operatorname{int}Q_0$, a feedback law, and
$R_0>0$. Lemma~\ref{lem:critical-constant} gives $z>1$ such that
\[
 b_0(1-z^{-2})>z.
\]
Choose $1<b_-<b_+$ sufficiently close to $b_0$ and a nondegenerate compact
coefficient interval $\mathcal A$ with $0\notin\mathcal A$ such that
\[
 Q_1=\mathcal A\times[b_-,b_+]
 \subset\operatorname{int}Q_0,
 \qquad
 b_- -\frac{b_+}{z^2}>z.
\]
By Proposition~\ref{prop:escape}, there is a constant $L_\star\ge1$ depending
only on $Q_1,W$, and $z$. Choose
$L_0\ge\max\{L_\star,\log(1\vee R_0)\}$. Then there exist
$p^\star\in Q_1\subset Q_0$ and a disturbance sequence satisfying
\eqref{eq:disturbance} such that $y_0=e^{L_0}\ge R_0$ and
$\log|y_t|\ge L_0z^t$ for every $t\ge0$. Thus
Theorem~\ref{thm:main}(ii) holds with $c=L_0$ and $T_0=1$.
\end{proof}

\section{Conclusion}\label{sec:conclusion}

This paper characterizes local robust stabilizability when both the coefficient
and growth exponent of the scalar nonlinear drift are unknown. The critical
exponent is $3\sqrt{3}/2$ under continuous exponent uncertainty and $4$ when
the exponent belongs to a known finite set and the coefficient ranges over a
nondegenerate compact interval. Consequently, a finite exponent set may be
stabilizable while its interval hull is not, even when their common diameter is
arbitrarily small. Thus the stabilizability limit depends on the structure of
the exponent uncertainty, not on its diameter alone.

\appendix

\section{Auxiliary results for Theorem~\ref{thm:main}(i)}
\label{app:positive-proofs}

This appendix proves the drift comparisons and the lemma for the record
recurrences used in Proposition~\ref{prop:away-zero-core} and verifies the
measurability statement in Remark~\ref{rem:bounded-random-disturbances}.

\subsection{Drift comparisons and record bounds}

\begin{lemma}\label{lem:two-scale}
Fix $p^\star=(a,b)$ and $p'=(a',b')$ with $aa'>0$. Let
$x_0,x_1\ne0$, set $L_i=\log|x_i|$, and suppose $L_0\ne L_1$ and
\[
 |G_{p'}(x_i)-G_{p^\star}(x_i)|\le2W,
 \qquad
 |G_{p^\star}(x_i)|\ge4W,
 \quad i=0,1.
\]
For $x\ne0$, put $L=\log|x|$ and define
\[
 E(x)\defeq
 \frac{|L_1-L|}{|L_1-L_0|}\frac{4W}{|G_{p^\star}(x_0)|}
 +
 \frac{|L-L_0|}{|L_1-L_0|}\frac{4W}{|G_{p^\star}(x_1)|}.
\]
Then
\begin{equation}\label{eq:two-state-log-ratio}
 \left|\log\frac{|G_{p'}(x)|}{|G_{p^\star}(x)|}\right|
 \le E(x).
\end{equation}
If $E(x)\le1$, then
\begin{equation}\label{eq:two-state-drift-error}
 |G_{p'}(x)-G_{p^\star}(x)|
 \le e|G_{p^\star}(x)|E(x).
\end{equation}
\end{lemma}

\begin{proof}
Set
$\Psi(s)=\log(|a'|/|a|)+(b'-b)s$. The common coefficient sign gives
\[
 \frac{G_{p'}(x_i)}{G_{p^\star}(x_i)}=e^{\Psi(L_i)},
 \qquad i=0,1.
\]
The bounds
$|G_{p'}(x_i)-G_{p^\star}(x_i)|\le2W$ and
$|G_{p^\star}(x_i)|\ge4W$ give
\[
 \left|e^{\Psi(L_i)}-1\right|
 \le\frac{2W}{|G_{p^\star}(x_i)|}\le\frac12.
\]
Applying $|\log q|\le2|q-1|$ with
$q=e^{\Psi(L_i)}$ gives
$|\Psi(L_i)|\le4W/|G_{p^\star}(x_i)|$. Since $\Psi$ is affine,
\[
 \Psi(L)=
 \frac{L_1-L}{L_1-L_0}\Psi(L_0)
 +\frac{L-L_0}{L_1-L_0}\Psi(L_1).
\]
Taking absolute values in the interpolation identity and using the bounds on
$\Psi(L_0)$ and $\Psi(L_1)$ gives $|\Psi(L)|\le E(x)$. The identity
\[
 \log\frac{|G_{p'}(x)|}{|G_{p^\star}(x)|}=\Psi(L),
\]
therefore proves \eqref{eq:two-state-log-ratio}. If $E(x)\le1$, then
$|\Psi(L)|\le1$ and
\[
 |G_{p'}(x)-G_{p^\star}(x)|
 =|G_{p^\star}(x)|\,|e^{\Psi(L)}-1|
 \le e|G_{p^\star}(x)|E(x),
\]
which is \eqref{eq:two-state-drift-error}.
\end{proof}

The next lemma transfers an error bound from one state to another.

\begin{lemma}\label{lem:one-anchor-upper-face}
Let $\overline a,\overline b,\Lambda>0$, and let
$p=(a,b)$ and $p'=(a',b')$ satisfy
\[
 0<|a|,|a'|\le\overline a,
 \qquad aa'>0,
 \qquad 0<b\le b'\le\overline b.
\]
If $x_0,x\in\R$ satisfy
\[
 1\le|x_0|,|x|\le e^{\Lambda},
 \qquad
 |G_{p'}(x_0)-G_p(x_0)|\le2W,
\]
then
\begin{equation}\label{eq:one-anchor-upper-face}
 \begin{aligned}
 |G_{p'}(x)-G_p(x)|
 &\le 2W\exp\!\left\{\overline b
       \max\{\log|x|-\log|x_0|,0\}\right\}\\
 &\quad+\overline a e^{\overline b\Lambda}
       (\overline b-b)\Lambda e^{(\overline b-b)\Lambda}.
 \end{aligned}
\end{equation}
\end{lemma}

\begin{proof}
Put $d=b'-b$ and $\chi=|x|/|x_0|$.
Because $aa'>0$, $G_p(x)$ and $G_{p'}(x)$ have the same sign at
each nonzero state. The exact decomposition
\[
 |a'||x|^{b'}-|a||x|^b
 =\bigl(|a'||x_0|^{b'}-|a||x_0|^b\bigr)\chi^{b'}
 +|a||x|^b(\chi^d-1)
\]
separates the prediction error observed at $x_0$ from the change caused by
the exponent difference $d$. It therefore gives
\[
 |G_{p'}(x)-G_p(x)|
 \le2W\chi^{b'}+|a||x|^b|\chi^d-1|.
\]
$2W\chi^{b'}$ is at most
$2W\exp\{\overline b\max\{\log\chi,0\}\}$. Moreover,
$d\ge0$ and $|\log\chi|\le\Lambda$ give
\[
 \begin{aligned}
 |a||x|^b|\chi^d-1|
 =|a||x|^b|e^{d\log\chi}-1|\le \overline a e^{\overline b\Lambda}
       d|\log\chi|e^{d|\log\chi|}\le\overline a e^{\overline b\Lambda}
       (\overline b-b)\Lambda
       e^{(\overline b-b)\Lambda}.
 \end{aligned}
\]
This proves \eqref{eq:one-anchor-upper-face}.
\end{proof}

The following scalar inequality yields the critical exponent in the record
argument.

\begin{lemma}\label{lem:critical-constant}
For $b>0$,
\[
 b(1-z^{-2})<z
 \qquad\text{for every }z>1
\]
if and only if $b<3\sqrt{3}/2$. At $b=3\sqrt{3}/2$, equality holds at
$z=\sqrt{3}$. If $b>3\sqrt{3}/2$, there is a $z>1$ such that
$b(1-z^{-2})>z$.
\end{lemma}

\begin{proof}
For $z>1$, the inequality $b(1-z^{-2})\ge z$ is equivalent to
\[
 b\ge \frac{z^3}{z^2-1}.
\]
The derivative of the function on the right is
$z^2(z^2-3)/(z^2-1)^2$.
Its unique minimum on $(1,\infty)$ is attained at $z=\sqrt{3}$ and equals
$3\sqrt{3}/2$. The assertion for $b>3\sqrt{3}/2$ follows by continuity.
\end{proof}

\begin{lemma}
\label{lem:record-closure}
Fix $L_{\max},C_0,\Lambda_0\ge0$,
$0<\overline b\le3\sqrt3/2$, and $\kappa>0$.
There is a finite constant
$M=M(L_{\max},C_0,\Lambda_0,\overline b,\kappa)$ with the following
property. Let $N\ge0$, $b\in(0,\overline b]$, and let
$(L_n)_{n=0}^N$ satisfy $0\le L_0\le L_{\max}$.
For $0\le n<N$, assume
\begin{equation}\label{eq:abstract-record-basic}
 L_n+\log2<L_{n+1}\le\overline bL_n+C_0.
\end{equation}
For $2\le n<N$ with $L_{n-2}\ge\Lambda_0$, assume
\begin{equation}\label{eq:abstract-record-two-state}
 L_{n+1}
 \le b(L_n-L_{n-2})+\log(1+L_n)+C_0.
\end{equation}
For $1\le n<N$ with $L_{n-1}\ge\Lambda_0$ and
$\overline b-b\le\kappa e^{-\overline bL_n}/(1+L_n)$, assume
\begin{equation}\label{eq:abstract-record-one-state}
 L_{n+1}\le\overline b(L_n-L_{n-1})+C_0.
\end{equation}
Then $\max_{0\le n\le N}L_n\le M$.
\end{lemma}

\begin{proof}
Suppose that no such $M$ exists. Define
\[
 U_0=L_{\max},
 \qquad
 U_{n+1}=\overline bU_n+C_0,
 \quad n\ge0.
\]
The upper bound in \eqref{eq:abstract-record-basic} gives $L_n\le U_n$ by
induction. If the lengths of all sequences satisfying the hypotheses were at
most $N_*$, then $\max_{0\le n\le N_*}U_n$ would be the required uniform
bound. Hence, for every $m\ge1$, there is a sequence satisfying the hypotheses
with at least $m+1$ entries. Truncate it after its $m$th entry, and write its
exponent and levels as $b_m$ and
$(L_k^{(m)})_{k=0}^m$.

For each fixed $k$, $0\le L_k^{(m)}\le U_k$ whenever $m\ge k$. Apply the
Bolzano--Weierstrass theorem successively to
$b_m,L_0^{(m)},L_1^{(m)},\ldots$ and then take the diagonal subsequence. On
this subsequence, $b_m$ and every fixed coordinate $L_k^{(m)}$ converge.
After relabeling this subsequence, denote the limits by
$b_\infty\in[0,\overline b]$ and $X_k\ge0$, $k\ge0$:
\[
 b_m\to b_\infty,
 \qquad
 L_k^{(m)}\to X_k
 \quad\text{for every fixed }k.
\]
For every fixed $k$, letting $m\to\infty$ in the lower bound in
\eqref{eq:abstract-record-basic} gives $X_{k+1}\ge X_k+\log2$. Hence
$X_k\to\infty$ and $X_k>0$ for $k\ge1$.
Set
\[
 r_k=\frac{X_{k+1}}{X_k},\qquad k\ge1,
 \qquad
 r=\limsup_{k\to\infty}r_k.
\]
The upper bound in \eqref{eq:abstract-record-basic} gives
$1<r_k\le\overline b+C_0/X_k$, so $1\le r\le\overline b$. This is already a
contradiction if $\overline b<1$.

Suppose first that $b_\infty<3\sqrt3/2$. Since the sequence
$(X_k)_{k\ge0}$ diverges, choose $K\ge2$ such that
$X_{k-2}>\Lambda_0$ for every $k\ge K$. Fix $k\ge K$. For all sufficiently
large $m$, $L_{k-2}^{(m)}\ge\Lambda_0$. Letting $m\to\infty$ in
\eqref{eq:abstract-record-two-state} gives
\[
 X_{k+1}\le
 b_\infty(X_k-X_{k-2})+\log(1+X_k)+C_0.
\]
Divide by $X_k$ and take $\limsup_{k\to\infty}$. Using
$X_{k-2}/X_k=(r_{k-1}r_{k-2})^{-1}$ yields
\[
 r\le b_\infty(1-r^{-2}).
\]
If $r=1$ or $b_\infty=0$, the right-hand side is zero, which is impossible.
For $r>1$ and $0<b_\infty<3\sqrt3/2$, the inequality contradicts
Lemma~\ref{lem:critical-constant}.

It remains to consider $b_\infty=3\sqrt3/2$. Then
$\overline b=b_\infty$ and $b_m\to\overline b$. Choose $K\ge1$ such that
$X_{k-1}>\Lambda_0$ for every $k\ge K$, and fix $k\ge K$. For all sufficiently
large $m$, $L_{k-1}^{(m)}\ge\Lambda_0$ and
\[
 \overline b-b_m
 \le \kappa\frac{e^{-\overline bU_k}}{1+U_k}
 \le \kappa\frac{e^{-\overline bL_k^{(m)}}}{1+L_k^{(m)}}.
\]
Here the first inequality follows from $b_m\to\overline b$, and the second
from $L_k^{(m)}\le U_k$ and the monotonicity of
$e^{-\overline bx}/(1+x)$ on $[0,\infty)$. Thus
\eqref{eq:abstract-record-one-state} applies for all sufficiently large $m$.
Letting $m\to\infty$ gives
\[
 X_{k+1}\le\overline b(X_k-X_{k-1})+C_0,
 \qquad k\ge K.
\]
Dividing by $X_k$ and taking $\limsup_{k\to\infty}$ yields
\[
 r\le\overline b(1-r^{-1}).
\]
For $r=1$ this reads $1\le0$. For $r>1$, it implies
\[
 \overline b\ge\frac{r^2}{r-1}
 =4+\frac{(r-2)^2}{r-1}\ge4,
\]
contrary to $\overline b=3\sqrt3/2$. Since
$0\le b_\infty\le\overline b\le3\sqrt3/2$, the two cases are exhaustive.
\end{proof}

\subsection{Borel measurability of the controller}

\begin{lemma}\label{lem:borel-controller}
Let $\mathcal P\subset\R\times(0,\infty)$ be nonempty and compact. For each
$t\ge0$, let $\mu_t(y_0,\ldots,y_t)$ be the value of $u_t$ specified by
\eqref{eq:fiber-midpoint-controller}, with value zero when
$\mathcal P_t=\varnothing$. Then $\mu_t:\R^{t+1}\to\R$ is Borel measurable.
\end{lemma}

\begin{proof}
We argue by induction on $t$. Eq.~\eqref{eq:fiber-midpoint-controller}
gives $\mu_0=0$. Fix $t\ge1$, suppose that
$\mu_0,\ldots,\mu_{t-1}$ are Borel, and write
$\mathbf y=(y_0,\ldots,y_t)\in\R^{t+1}$. Define
\[
 \Gamma_t=\left\{(\mathbf y,p)\in\R^{t+1}\times\mathcal P:
 |y_{k+1}-\mu_k(y_0,\ldots,y_k)-G_p(y_k)|\le W,
 \ 0\le k<t\right\}.
\]
The induction hypothesis and the continuity of $(p,y)\mapsto G_p(y)$ imply
that $\Gamma_t$ is Borel. Its section at $\mathbf y$ is the compact set
$\mathcal P_t(\mathbf y)$ defined in \eqref{eq:feasible-set}. For every open
$O\subset\mathcal P$, the sections of
$\Gamma_t\cap(\R^{t+1}\times O)$ are $\sigma$-compact. The
Arsenin--Kunugui theorem \cite[Theorem~18.18]{Kechris1995} therefore shows that
\[
 \{\mathbf y:\mathcal P_t(\mathbf y)\cap O\ne\varnothing\}
\]
is Borel. Thus $\mathbf y\mapsto\mathcal P_t(\mathbf y)$ is a measurable compact-valued
correspondence on the Borel set where it is nonempty.

The measurable maximum theorem
\cite[Theorem~18.19]{AliprantisBorder2006}, applied to $(a,b)\mapsto b$,
shows that $\mathbf y\mapsto\mathcal F_t(\mathbf y)$ is a measurable compact-valued
correspondence on $\{\mathbf y:\mathcal P_t(\mathbf y)\ne\varnothing\}$.
A second application, to $(\mathbf y,p)\mapsto G_p(y_t)$ and its negative, shows
that the two extrema defining $m_t(\mathcal F_t)$ are Borel on
$\{\mathbf y:\mathcal P_t(\mathbf y)\ne\varnothing\}$. The set
$\{\mathbf y:\mathcal P_t(\mathbf y)\ne\varnothing\}$ and
$\{\mathbf y:\max_{1\le s\le t}|y_s|>2W\}$ are Borel. Hence the controller
\eqref{eq:fiber-midpoint-controller}, extended by $\mu_t(\mathbf y)=0$
when $\mathcal P_t(\mathbf y)=\varnothing$, is Borel measurable. Since
$\mu_0=0$, induction proves the claim for every $t\ge0$.
\end{proof}

\section{Proof of Theorem~\ref{thm:finite-set}}

\subsection{The case of a known exponent}

When the exponent is known, a past observation bounds the diameter of the
remaining coefficient interval and yields a direct recurrence for successive
new maxima of $|y_t|$.

\begin{proposition}\label{prop:scalar}
Fix $0<b<4$ and a nonempty compact interval
$\mathcal A\subset\R$. Initialize $\mathcal P_t$ in
\eqref{eq:feasible-set} by
$\mathcal A\times\{b\}$ and use
\begin{equation}\label{eq:full-midpoint-controller}
 u_t=
 \begin{cases}
  -m_t(\mathcal P_t),&\mathcal P_t\ne\varnothing,\\
  0,&\mathcal P_t=\varnothing.
 \end{cases}
\end{equation}
This controller robustly stabilizes \eqref{eq:system}.
For every $Y\ge0$, its state bound is uniform over $a\in\mathcal A$,
$|y_0|\le Y$, and all disturbance sequences satisfying
\eqref{eq:disturbance}. The bound depends on $Y,W,b$ and on $\mathcal A$
only through $\max_{a\in\mathcal A}|a|$.
\end{proposition}

\begin{proof}
For $t\ge0$, let
\[
 \mathcal A_t\defeq\{a':(a',b)\in\mathcal P_t\}.
\]
Each constraint in \eqref{eq:feasible-set} is a closed interval in $a'$.
Thus $\mathcal A_t$ is a compact interval. It contains the true coefficient
and is therefore nonempty. Let $d_t=\operatorname{diam}(\mathcal A_t)$.
Since $\mathcal P_t=\mathcal A_t\times\{b\}$,
$m_t(\mathcal P_t)$ is $\phi_b(y_t)$ times the midpoint of $\mathcal A_t$.
The distance from the true coefficient to this midpoint is at most $d_t/2$,
and therefore
\begin{equation}\label{eq:scalar-diameter-state}
 |y_{t+1}|\le W+\frac{d_t}{2}|y_t|^b.
\end{equation}
If $k<t$ and $y_k\ne0$, feasibility of any $a',a''\in\mathcal A_t$ gives
\[
 |a'-a''|\,|y_k|^b
 \le |y_{k+1}-u_k-a'\phi_b(y_k)|
    +|y_{k+1}-u_k-a''\phi_b(y_k)|
 \le2W.
\]
Taking the supremum over $a',a''\in\mathcal A_t$ yields
\begin{equation}\label{eq:scalar-feasible-diameter}
 d_t\le\frac{2W}{|y_k|^b}.
\end{equation}

Let $\tau_0=\inf\{t\ge0:|y_t|>1\}$. If $\tau_0=\infty$, the state is
bounded by one. Suppose that $\tau_0<\infty$, set $X_0=|y_{\tau_0}|$, and
define $A_{\max}=\max_{a'\in\mathcal A}|a'|$.
If $\tau_0=0$, then $X_0\le Y$. If $\tau_0>0$, then
$|y_{\tau_0-1}|\le1$, while both
$|a\phi_b(y_{\tau_0-1})|$ and $|u_{\tau_0-1}|$ are at most $A_{\max}$. Thus
$X_0\le\max\{Y,2A_{\max}+W\}$.

For each $n\ge0$ with $\tau_n<\infty$, define the next record time by
\[
 \tau_{n+1}=\inf\{t>\tau_n:|y_t|>X_n\}.
\]
Whenever $\tau_{n+1}<\infty$, set
$X_{n+1}=|y_{\tau_{n+1}}|$ and $L_{n+1}=\log X_{n+1}$, with
$L_0=\log X_0$. The finite record values satisfy $X_{n+1}>X_n$. For every
finite $\tau_n$, the definition gives
$|y_t|\le X_n$ for $\tau_n\le t<\tau_{n+1}$. In particular, if
$\tau_1<\infty$, the state equation at time $\tau_1-1$ gives
\[
 X_1\le
 2A_{\max}\max\{Y,2A_{\max}+W\}^b+W.
\]

Fix $n\ge1$ such that $\tau_{n+1}<\infty$, and set
$t=\tau_{n+1}-1$. Since $\tau_{n-1}<t$, the transition from
$y_{\tau_{n-1}}$ is among the constraints defining $\mathcal A_t$.
Moreover, $|y_t|\le X_n$. Applying
\eqref{eq:scalar-feasible-diameter} with $k=\tau_{n-1}$ in
\eqref{eq:scalar-diameter-state} gives
\[
 X_{n+1}\le
 W\left[1+\left(\frac{X_n}{X_{n-1}}\right)^b\right]
 \le (1+2W)\exp\{b(L_n-L_{n-1})\}.
\]
Thus, with $c\defeq\log(1+2W)$,
\begin{equation}\label{eq:scalar-log-record-direct}
 L_{n+1}\le b(L_n-L_{n-1})+c,
 \qquad n\ge1,\quad \tau_{n+1}<\infty.
\end{equation}

Define
\[
 \gamma\defeq\inf_{r\ge1}\{r-b(1-r^{-1})\}.
\]
The derivative of $r-b+b/r$ is $1-b/r^2$. For $0<b\le1$, the function is
nondecreasing on $[1,\infty)$ and its minimum is $1$. For $1<b<4$, its
minimum is attained at $r=\sqrt b$ and equals $2\sqrt b-b>0$. Hence
$\gamma>0$.
The bounds on $X_0$ and $X_1$ allow us to choose
$\Lambda>1$, depending only on $Y,W,b$, and $A_{\max}$, such that
\[
 L_0<\Lambda,
 \qquad L_1<\Lambda\quad\text{if }\tau_1<\infty,
 \qquad
 \frac{c}{\Lambda}\le\min\left\{1,\frac{\gamma}{2}\right\}.
\]

Let $j$ be the first index with $L_j\ge\Lambda$. If no such index exists,
then $|y_t|\le e^\Lambda$ for all $t$. Hence assume that $j$ exists. Then
$j\ge2$. By the minimality of $j$, $L_{j-1}<\Lambda$, while
$L_{j-2}>0$. Eq.~\eqref{eq:scalar-log-record-direct} at $n=j-1$
therefore gives $L_j<b\Lambda+c$. If $\tau_{j+1}=\infty$, then
$|y_t|\le X_j\le e^{b\Lambda+c}$ after $\tau_j$. Hence assume that
$\tau_{j+1}<\infty$. For every finite record with index $n\ge j+1$, define
$\rho_n=L_n/L_{n-1}$. The record values are strictly increasing, so
$\rho_n>1$. Dividing \eqref{eq:scalar-log-record-direct} at $n=j$ by $L_j$
and using $c/L_j\le c/\Lambda\le1$ gives
$\rho_{j+1}\le b(1-L_{j-1}/L_j)+1\le b+1$.
If $n\ge j+1$ and $\tau_{n+1}<\infty$, then
\[
 \rho_{n+1}
 \le b(1-\rho_n^{-1})+\frac{c}{L_n}
 \le b(1-\rho_n^{-1})+\frac{\gamma}{2}
 \le\rho_n-\frac{\gamma}{2}.
\]
Choose an integer $N>2b/\gamma$. If $\tau_{j+N+1}$ were finite, iterating this
inequality would give $\rho_{j+N+1}<1$, a contradiction. Hence at most $N$
records occur after $\tau_j$. Since
$\rho_{n+1}\le\rho_n-\gamma/2$, every finite ratio after $\tau_j$ is at
most $\rho_{j+1}\le b+1$. Consequently,
\[
 L_{j+s}\le(b+1)^sL_j
 \le(b+1)^N(b\Lambda+c),
 \qquad 0\le s\le N,\quad \tau_{j+s}<\infty.
\]
Since $|y_t|\le X_n$ for $\tau_n\le t<\tau_{n+1}$, exponentiating the last
display and including the cases in which $j$ or $\tau_{j+1}$ does not exist
gives
\[
 \sup_{t\ge0}|y_t|
 \le \exp\!\left\{
      \max\bigl\{\Lambda,(b+1)^N(b\Lambda+c)\bigr\}\right\}.
\]
The quantities on the right depend only on $Y,W,b$, and $A_{\max}$, as
claimed.
\end{proof}

\subsection{Finite exponent sets}

The following lemma extends Proposition~\ref{prop:scalar} from a known
exponent to finitely many possible exponents when the coefficient interval is
separated from zero.

\begin{lemma}\label{lem:finite-away-zero}
Let $\mathcal B\subset(0,\infty)$ be nonempty and finite with
$\max_{b\in\mathcal B}b<4$, and let
$0<\underline a\le\overline a<\infty$.
For every $Y\ge0$, there is a constant
$C=C(Y,\underline a,\overline a,\mathcal B,W)<\infty$ such that, for every
compact interval $I\subset\R\setminus\{0\}$ satisfying
$\underline a\le|a|\le\overline a$ for all $a\in I$, the controller
\eqref{eq:full-midpoint-controller}, with $\mathcal P_t$ initialized by
$I\times\mathcal B$, satisfies
\[
 \sup_{\substack{(a,b)\in I\times\mathcal B,\ |y_0|\le Y,\ t\ge0\\
                   (w_k)_{k\ge1}:\ \sup_{k\ge1}|w_k|\le W}}
 |y_t|\le C.
\]
\end{lemma}

\begin{proof}
If $\mathcal B$ is a singleton, the result follows from
Proposition~\ref{prop:scalar}. Assume henceforth that $\mathcal B$ contains at
least two elements. Let $p^\star=(a^\star,b^\star)\in I\times\mathcal B$,
$|y_0|\le Y$, and let the disturbance sequence satisfy
\eqref{eq:disturbance}. Let $\mathcal P_t$ be the resulting feasible sets.
Since $p^\star\in\mathcal P_t$ for every $t$, these sets are nonempty. All
coefficients in $I$ have a common sign. Put
$\overline b=\max_{b\in\mathcal B}b$. Since $\mathcal B$ is finite,
there is an $R_*>1$, depending only on
$\underline a,\overline a,\mathcal B$, and $W$, such that
\begin{equation}\label{eq:finite-uniform-separation}
 \underline aR^{b_k}-\overline aR^{b_j}>2W
 \quad\text{for every }R\ge R_*
 \text{ and every }b_j<b_k\text{ in }\mathcal B.
\end{equation}
Indeed, each difference on the left tends to infinity as $R\to\infty$, and
only finitely many pairs occur.

Define
\[
 \tau=\inf\{t\ge0:|y_t|\ge R_*\}.
\]
If $\tau=\infty$, then the trajectory is bounded by $R_*$. Suppose that
$\tau<\infty$. If $\tau>0$, then
$|y_{\tau-1}|<R_*$. Every $p\in\mathcal P_{\tau-1}$ then satisfies
$|G_p(y_{\tau-1})|\le\overline aR_*^{\overline b}$, and the same bound holds
for $|m_{\tau-1}(\mathcal P_{\tau-1})|$. Define
\[
 Y_0=\max\{Y,R_*,2\overline aR_*^{\overline b}+W\},
 \qquad
 Y_1=\max\{Y_0,2\overline aY_0^{\overline b}+W\}.
\]
When $\tau>0$, the state equation at time $\tau-1$ gives
$|y_\tau|\le2\overline aR_*^{\overline b}+W\le Y_0$. When $\tau=0$,
$|y_\tau|\le Y\le Y_0$. At time $\tau$,
$\max_{p\in\mathcal P_\tau}|G_p(y_\tau)|$ and
$|m_\tau(\mathcal P_\tau)|$ are at most
$\overline aY_0^{\overline b}$. Hence
$|y_{\tau+1}|\le2\overline aY_0^{\overline b}+W\le Y_1$.

After the transition from $y_\tau$ has been observed, two distinct exponents
cannot remain feasible. Indeed, if
$(a_j,b_j),(a_k,b_k)\in\mathcal P_{\tau+1}$ with $b_j<b_k$, feasibility gives
\[
 |G_{(a_j,b_j)}(y_\tau)-G_{(a_k,b_k)}(y_\tau)|\le2W.
\]
Since all coefficients have one sign, the left-hand side is at least
\[
 \underline a|y_\tau|^{b_k}
 -\overline a|y_\tau|^{b_j}>2W
\]
by \eqref{eq:finite-uniform-separation}, a contradiction. Hence every
parameter in $\mathcal P_{\tau+1}$ has exponent $b^\star$. Let
\[
 I_{\tau+1}=\{a'\in I:(a',b^\star)\in\mathcal P_{\tau+1}\}.
\]
For the fixed exponent $b^\star$, every constraint in
\eqref{eq:feasible-set} is a closed interval in $a'$. Hence
$I_{\tau+1}$ is a nonempty compact interval containing $a^\star$, and
\[
 \mathcal P_{\tau+1}=I_{\tau+1}\times\{b^\star\}.
\]
All observations through time $\tau$ are encoded in $\mathcal P_{\tau+1}$.
Starting at time $\tau+1$, apply the controller
\eqref{eq:full-midpoint-controller} with initial state $y_{\tau+1}$ and
initial parameter set $\mathcal P_{\tau+1}$.
At time $\tau+1$, this process and the original process have the same state and
feasible set. If the two states and feasible sets agree at time $t$, then the
inputs agree by \eqref{eq:full-midpoint-controller}. The state equation,
with the same $p^\star$ and $w_{t+1}$, gives the same next state, and
\eqref{eq:feasible-set} then gives the same feasible set at time $t+1$.
Thus the two processes have the same $y_t$, $u_t$, and $\mathcal P_t$ for every
$t\ge\tau+1$. Applying
Proposition~\ref{prop:scalar} from the state $y_{\tau+1}$ gives a state bound
depending only on $Y_1,W,b^\star$, and $\overline a$.
Before time $\tau+1$, the state is bounded by $Y_1$. Taking the maximum over
the finite set $\mathcal B$ gives a bound depending only on
$Y,\underline a,\overline a,\mathcal B$, and $W$.
\end{proof}

\begin{proof}[Proof of Theorem~\ref{thm:finite-set}]
We first prove sufficiency. Suppose that $\max_{b\in\mathcal B}b<4$ and fix
$Y\ge0$. Consider any true parameter, initial state, and disturbance sequence
in \eqref{eq:uniform-state-bound} with
$\mathcal P=\mathcal A\times\mathcal B$. Initialize \eqref{eq:feasible-set}
with this set and use zero input until
$t_{\rm a}=\inf\{s\ge1:|y_s|>2W\}$. If $t_{\rm a}=\infty$, then
$|y_t|\le2W$ for $t\ge1$.

Suppose that $t_{\rm a}<\infty$ and apply
Lemma~\ref{lem:first-exit-reduction} to
$\mathcal P=\mathcal A\times\mathcal B$. Set
$y_{\rm a}=y_{t_{\rm a}}$ and define the coefficient interval
\[
 \mathcal A_{t_{\rm a}}=
 \left[
 \min\{a:(a,b)\in\mathcal P_{t_{\rm a}}\},
 \max\{a:(a,b)\in\mathcal P_{t_{\rm a}}\}
 \right].
\]
The interval $\mathcal A_{t_{\rm a}}$ is compact, does not contain zero,
contains the true coefficient, and satisfies the coefficient bounds in
\eqref{eq:first-exit-reduction}.

From time $t_{\rm a}$ onward, use the controller from
Lemma~\ref{lem:finite-away-zero} on
$\mathcal A_{t_{\rm a}}\times\mathcal B$, with
initial state $y_{\rm a}$. The interval $\mathcal A_{t_{\rm a}}$ and the state
$y_{\rm a}$ are determined by the history at time $t_{\rm a}$, so the
resulting law is causal. The bounds in
\eqref{eq:first-exit-reduction} and Lemma~\ref{lem:finite-away-zero} bound the
state from time $t_{\rm a}$ onward. Together with $|y_0|\le Y$ and
$|y_t|\le2W$ for $1\le t<t_{\rm a}$, this proves sufficiency.

For necessity, suppose that $\max_{b\in\mathcal B}b\ge4$ and choose
$b_j\in\mathcal B$ with $b_j\ge4$. The nondegenerate interval $\mathcal A$
contains a nondegenerate compact subinterval
$I\subset\mathcal A\setminus\{0\}$. Apply
\cite[Theorem~2.3]{LiXie2006} on $I$ with $f=\phi_{b_j}$. Its lower growth
condition holds with equality because $|\phi_{b_j}(y)|=|y|^{b_j}$. Its
stability requirement is the same global boundedness requirement as
Definition~\ref{def:global-robust}. Hence $I\times\{b_j\}$ is not robustly
stabilizable. Robust stabilization of $\mathcal A\times\mathcal B$ would imply
stabilization of
$I\times\{b_j\}$, which is a contradiction.
\end{proof}

\section{Proofs of the extensions}

\subsection{Proof of Corollary~\ref{cor:shape-factor}}

Put $G_p^h(y)=h(y)G_p(y)$. We first consider a nonempty compact set
\[
 \mathcal P\subset
 \{(a,b):\underline a\le|a|\le\overline a,
                 \ b_-\le b\le\overline b\le3\sqrt3/2\},
\]
where $\underline a>0$ and $b_->0$. Define the feasible sets and the subsets
on which the exponent is maximal as in Section~\ref{sec:positive}, with
$G_p^h$ in place of $G_p$. Since $h(y_t)>0$ is known and does not depend on
$p$, the midpoint of $G_p^h(y_t)$ over a nonempty compact set $\mathcal S$ is
\[
 h(y_t)m_t(\mathcal S).
\]
Fix $Y\ge0$, $p^\star=(a^\star,b^\star)\in\mathcal P$, and $|y_0|\le Y$,
and let the disturbance sequence satisfy \eqref{eq:disturbance}.
Use the control $u_t=-h(y_t)m_t(\mathcal F_t)$. Feasibility at any previously
observed state $x=y_k$, $0\le k<t$, gives the first inequality below for every
$p\in\mathcal F_t$. The control law and the system equation give the second:
\begin{equation}\label{eq:shape-comparison-bounds}
 |G_p(x)-G_{p^\star}(x)|\le\frac{2W}{\underline h},
 \qquad
 |y_{t+1}|\le W+\overline h
 \max_{p\in\mathcal F_t}|G_p(y_t)-G_{p^\star}(y_t)|.
\end{equation}
When $|y_t|\le1$, every $|G_p^h(y_t)|$ with $p\in\mathcal P_t$ is at most
$\overline h\,\overline a$, and therefore
\begin{equation}\label{eq:shape-small-state-bound}
 \max\{1,|y_{t+1}|\}
 \le \max\{1,2\overline h\,\overline a+W\}.
\end{equation}

Set $\tau_0=0$ and $L_0=\log\max\{1,|y_0|\}$. For each $n\ge0$ with
$\tau_n<\infty$, define $\tau_{n+1}$ by \eqref{eq:record-times}. If
$\tau_{n+1}<\infty$, set
$L_{n+1}=\log\max\{1,|y_{\tau_{n+1}}|\}$. At a previously observed state,
the first bound in \eqref{eq:shape-comparison-bounds} allows
Lemmas~\ref{lem:two-scale} and \ref{lem:one-anchor-upper-face} to be applied
with $W/\underline h$ in place of $W$. At the current state, the second bound
multiplies the resulting prediction error by at most $\overline h$. Together
with $|G_p^h(y)|\le\overline h\,\overline a|y|^{\overline b}$ and
\eqref{eq:shape-small-state-bound}, the derivation in
Proposition~\ref{prop:away-zero-core} yields a nonnegative constant $C_h$ and a
level $\Lambda_0^h\ge0$. They depend only on
$\underline a,\overline a,b_-,\overline b,W,\underline h$, and $\overline h$.
Set
\begin{equation}\label{eq:shape-upper-face-kappa}
  \kappa_h=\min\left\{1,
  \frac{1}{e\overline h\,\overline a\,2^{\overline b}}\right\}.
\end{equation}
 The resulting record estimates are
 \begin{align}
  L_{n+1}&\le\overline bL_n+C_h,
  \label{eq:shape-coarse-record}\\
  L_{n+1}&\le
  b^\star(L_n-L_{n-2})+\log(1+L_n)+C_h,
  \label{eq:shape-two-state-record}\\
  L_{n+1}&\le\overline b(L_n-L_{n-1})+C_h.
  \label{eq:shape-one-state-record}
 \end{align}
 Eq.~\eqref{eq:shape-coarse-record} holds for $n\ge0$ whenever
 $\tau_{n+1}<\infty$. Eq.~\eqref{eq:shape-two-state-record} holds for
 $n\ge2$ whenever $\tau_{n+1}<\infty$ and
 $L_{n-2}\ge\Lambda_0^h$. Eq.~\eqref{eq:shape-one-state-record} holds for
 $n\ge1$ whenever $\tau_{n+1}<\infty$, $L_{n-1}\ge\Lambda_0^h$, and
 \[
  \overline b-b^\star
   \le\kappa_h\frac{e^{-\overline bL_n}}{1+L_n}.
 \]
 Since $L_n+\log2\le1+L_n$, this condition gives
 \[
  \overline h\,\overline a e^{\overline b(L_n+\log2)}
  (\overline b-b^\star)(L_n+\log2)
  e^{(\overline b-b^\star)(L_n+\log2)}
   \le e\,\overline h\,\overline a\,2^{\overline b}\kappa_h\le1.
 \]
 By \eqref{eq:shape-upper-face-kappa}, the second summand on the right of
 \eqref{eq:one-anchor-upper-face}, after multiplication by $\overline h$, is at
 most one. The first summand, with $W/\underline h$ in place of $W$ and then
 multiplied by $\overline h$, is absorbed into $C_h$.

The hypotheses of Lemma~\ref{lem:record-closure} follow from
 \eqref{eq:record-increment}, \eqref{eq:shape-coarse-record},
 \eqref{eq:shape-two-state-record}, and \eqref{eq:shape-one-state-record}, with
 $L_{\max}=\log(\max\{1,Y\})$, $C_0=C_h$, $\Lambda_0=\Lambda_0^h$, and
 $\kappa=\kappa_h$ from \eqref{eq:shape-upper-face-kappa}. The lemma therefore
 gives a common upper bound for all $L_n$, and \eqref{eq:between-records}
 gives a uniform state bound.

Now let $\mathcal P$ be any nonempty compact parameter set with
$\overline b=\max_{(a,b)\in\mathcal P}b\le3\sqrt3/2$, and set
$\overline a=\max_{(a,b)\in\mathcal P}|a|$ and
$b_-=\min_{(a,b)\in\mathcal P}b>0$. Fix
$p^\star\in\mathcal P$, $|y_0|\le Y$, and a disturbance sequence satisfying
\eqref{eq:disturbance}, and use zero input until
$t_{\rm a}=\inf\{s\ge1:|y_s|>2W\}$. If $t_{\rm a}=\infty$, then
$|y_s|\le2W$ for $s\ge1$. Otherwise, let $S=\max\{1,Y,2W\}$. Since
$u_{t_{\rm a}-1}=0$ and $|y_{t_{\rm a}}|>2W$, feasibility gives, for every
$p=(a,b)\in\mathcal P_{t_{\rm a}}$,
$|G_p^h(y_{t_{\rm a}-1})|\ge|y_{t_{\rm a}}|-W>W$.
Because $|G_p^h(y_{t_{\rm a}-1})|
\le |a|\overline hS^{\overline b}$, it follows that
$|a|>W/(\overline hS^{\overline b})$. All values
$G_p^h(y_{t_{\rm a}-1})$, $p\in\mathcal P_{t_{\rm a}}$, are within $W$ of
$y_{t_{\rm a}}$ and therefore have the same nonzero sign. The true transition,
for which the input is zero, also gives
$|y_{t_{\rm a}}|\le\overline a\,\overline hS^{\overline b}+W$.
Starting at time $t_{\rm a}$, initialize the feasible sets by
$\mathcal P_{t_{\rm a}}$ and apply
$u_t=-h(y_t)m_t(\mathcal F_t)$. The uniform bound established above applies
because every parameter in $\mathcal P_{t_{\rm a}}$ satisfies
$W/(\overline hS^{\overline b})\le|a|\le\overline a$, while
$|y_{t_{\rm a}}|\le\overline a\,\overline hS^{\overline b}+W$. Both
$t_{\rm a}$ and
$\mathcal P_{t_{\rm a}}$ are determined by the observed history, so the
switched law is causal. The resulting state bound is uniform over
$p^\star\in\mathcal P$, $|y_0|\le Y$, and all disturbance sequences satisfying
\eqref{eq:disturbance}.

We next prove the converse. Fix $p_0=(a_0,b_0)$ with
$b_0>3\sqrt3/2$, a compact rectangle $Q_0$ containing $p_0$ in its interior,
an arbitrary feedback law, and $R_0>0$. Lemma~\ref{lem:critical-constant}
and continuity allow us to choose, inside $Q_0$, a compact rectangle
$Q=\mathcal A\times[b_-,b_+]$ whose coefficient interval does not contain
zero, and a number $z>1$ such that
\[
 b_- -\frac{b_+}{z^2}>z.
\]
Let $\sigma$ be the sign of the coefficients in $\mathcal A$ and write
\[
 T_\sigma(a,b)=(\log|a|,b),
 \qquad
 T_\sigma^{-1}(\alpha,b)=(\sigma e^\alpha,b).
\]
For $\eta=(\alpha,b)$ and $L\in\R$, put $\ell_\eta(L)=\alpha+bL$. The image
of $Q$ under $T_\sigma$ is
\[
 \widetilde Q\defeq T_\sigma(Q)
 =[\alpha_-,\alpha_+]\times[b_-,b_+].
\]
For $\eta=(\alpha,b)$, $y\ne0$, and $L=\log|y|$, write
\[
 G_\eta^h(y)
 \defeq G_{T_\sigma^{-1}(\eta)}^h(y)
 =h(y)\sigma\operatorname{sgn}(y)e^{\ell_\eta(L)}.
\]
Set
\[
 \alpha_{\rm m}=\frac{\alpha_-+\alpha_+}{2},
 \qquad
 b_{\rm m}=\frac{b_-+b_+}{2},
 \qquad
 d=\frac{\alpha_+-\alpha_-}{8}.
\]
Let $L_0>0$ and set $y_0=e^{L_0}$. After $u_0$ is selected,
choose $\eta_0$ from
\[
 \{(\alpha_{\rm m}-d,b_{\rm m}),
   (\alpha_{\rm m}+d,b_{\rm m})\}
\]
to maximize $|u_0+G_\eta^h(y_0)|$, and set
$y_1=u_0+G_{\eta_0}^h(y_0)$. The inequality
$\max\{|u+x|,|u+y|\}\ge|x-y|/2$ gives
\[
 |y_1|
 \ge \underline h\,
       \frac{e^d-e^{-d}}2e^{\alpha_{\rm m}+b_{\rm m}L_0}.
\]
For all sufficiently large $L_0$, the right-hand side is at least
$e^{zL_0}$. Hence $y_1\ne0$, and
\[
 L_1\defeq\log|y_1|\ge zL_0.
\]

Set $v_1=(1,-L_0^{-1})$. Since $\ell_{v_1}(L_0)=0$, the parameters
$\eta_0$ and $\eta_0+dv_1$ give the same drift at $y_0$. For all sufficiently
large $L_0$, both parameters belong to $\widetilde Q$. After $u_1$ is
selected, choose $\eta_1$ from $\{\eta_0,\eta_0+dv_1\}$ to maximize
$|u_1+G_\eta^h(y_1)|$, and set $y_2=u_1+G_{\eta_1}^h(y_1)$. The two values of
$\ell_\eta(L_1)$ differ by at least $d(z-1)$. The inequality
$\max\{|u+x|,|u+y|\}\ge|x-y|/2$ gives
\[
 |y_2|
 \ge \frac{\underline h}{2}e^{\alpha_-+b_-L_1}
       \bigl(e^{d(z-1)}-1\bigr).
\]
For all sufficiently large $L_0$, this lower bound is at least $e^{zL_1}$.
Hence $y_2\ne0$, and
\[
 L_2\defeq\log|y_2|\ge zL_1.
\]
Write $\eta_1=(\alpha_1,b_1)$. The definitions of $\eta_0$ and $\eta_1$ give
\eqref{eq:escape-initial-parameter-bounds}.

For this proof, define
\[
 s_k=\frac{W(z-1)}{8z\overline h}e^{-\alpha_+-b_+L_k},
 \qquad k\ge0.
\]
Choose $L_0$ large enough that the conditions in
\eqref{eq:escape-large-initial-level} hold after substituting
$s_0=W(z-1)e^{-\alpha_+-b_+L_0}/(8z\overline h)$, and
that
\[
 \log\!\left(\frac{W(z-1)^2\underline h}{16z\overline h}\right)
 +\alpha_--\alpha_++
 \left(b_--\frac{b_+}{z^2}-z\right)L_0\ge0.
\]

Fix $n\ge2$ and suppose that $L_j\ge zL_{j-1}$ has been established for
$1\le j\le n$. Use the direction $v_n$ in
\eqref{eq:escape-update-direction}.
With the present definition of $s_k$, the calculation in
\eqref{eq:escape-candidate-displacement} shows that both
\[
 \eta_{n-1},
 \qquad
 \eta_{n-1}+s_{n-2}v_n
\]
belong to $\widetilde Q$. After $u_n$ is known, choose between them to maximize
$|u_n+G_\eta^h(y_n)|$, denote the selected parameter by $\eta_n$, and set
$y_{n+1}=u_n+G_{\eta_n}^h(y_n)$. Since $h(y_n)\ge\underline h$,
\eqref{eq:escape-update-properties} gives
\[
 |y_{n+1}|
 \ge \frac{z-1}{2}\underline h\,s_{n-2}
      e^{\alpha_-+b_-L_n}
 =\frac{W(z-1)^2\underline h}{16z\overline h}
  e^{\alpha_--\alpha_++b_-L_n-b_+L_{n-2}},
\]
so $y_{n+1}\ne0$. Define $L_{n+1}=\log|y_{n+1}|$. Then
\[
 L_{n+1}
 \ge\log\!\left(\frac{W(z-1)^2\underline h}{16z\overline h}\right)
     +\alpha_--\alpha_++b_-L_n-b_+L_{n-2}
 \ge zL_n.
\]
This proves the induction step. The growth inequalities and the first
condition in \eqref{eq:escape-large-initial-level} give
$s_{n+1}\le s_n/2$ for $n\ge0$. By
\eqref{eq:escape-update-direction}, the magnitudes of the two coordinates of
$v_n$ are at most $z/(z-1)$ and $1/((z-1)L_{n-2})$, respectively. Thus the
sums of the absolute changes in both coordinates of $\eta_n$ are finite. Hence
$(\eta_n)_{n\ge0}$ converges to some $\eta^\star\in\widetilde Q$.

The initialization and the recursive choices give
\[
 \ell_{\eta_1-\eta_0}(L_0)=0,
 \qquad
 \ell_{\eta_{k+1}-\eta_k}(L_k)=0,
 \quad k\ge1,
\]
where the second equality follows from
\eqref{eq:escape-update-properties}. For all $k\ge0$ and $n\ge k+2$, the
update direction satisfies
$|\ell_{v_n}(L_k)|\le z/(z-1)$. Summing the remaining updates and using
$s_{j+1}\le s_j/2$ for $j\ge0$ gives
\[
 |\ell_{\eta^\star-\eta_k}(L_k)|
 \le\frac{2z}{z-1}s_k
 =\frac{W}{4\overline h}e^{-\alpha_+-b_+L_k}.
\]
Consequently,
\[
 |G_{\eta_k}^h(y_k)-G_{\eta^\star}^h(y_k)|
 \le\overline h e^{\alpha_++b_+L_k}
      |\ell_{\eta_k-\eta^\star}(L_k)|
 \le\frac W4.
\]
Set $p^\star=T_\sigma^{-1}(\eta^\star)\in Q$ and, for $k\ge0$, define
\[
 w_{k+1}=G_{\eta_k}^h(y_k)-G_{\eta^\star}^h(y_k).
\]
Then $|w_{k+1}|\le W$ and
\[
 y_{k+1}=G_{p^\star}^h(y_k)+u_k+w_{k+1},
 \qquad k\ge0.
\]
Thus the fixed parameter $p^\star$ and the fixed disturbance sequence
$(w_{k+1})_{k\ge0}$ generate the escaping trajectory. Since $Q\subset Q_0$
and the initial level can be chosen so that $e^{L_0}\ge R_0$, $Q_0$ is not
robustly stabilizable. Finally, when
$b_0=3\sqrt3/2$, every compact rectangle containing $p_0$ in its interior
contains a point with a larger exponent. The positive result and the converse
just proved give the local classification.

\subsection{Proof of Corollary~\ref{cor:parameter-map}}

For a feedback law $\mu=(\mu_t)_{t\ge0}$, define
\[
 \widetilde\mu_t(y_0,\ldots,y_t)
 =\mu_t(y_0,\ldots,y_t)+g_0(y_t),
 \qquad t\ge0.
\]
Because $g_0$ is known and $y_t$ is observed when the input is chosen,
$\widetilde\mu$ is causal. Subtracting $g_0(y_t)$ defines the inverse, so the
transformation is a bijection between feedback laws. If
$u_t=\mu_t(y_0,\ldots,y_t)$ is used in \eqref{eq:general-parameter-map} and
$\widetilde u_t=\widetilde\mu_t(y_0,\ldots,y_t)$ is used in
\eqref{eq:system} with parameter $p=H(\xi)$, then
$\widetilde u_t=u_t+g_0(y_t)$ and the two state equations are identical. For
every compact $K\subset\Xi$, the two closed loops therefore have the same
state trajectory for each initial state and disturbance sequence. Hence robust
stabilization over $K$ is equivalent to robust stabilization of the
system~\eqref{eq:system} over $H(K)$, and the uniform estimate
\eqref{eq:uniform-state-bound} is preserved.

Continuity of $H$ and compactness of $\Xi$ imply that $H(\Xi)$ is compact.
Part~(i) now follows from Theorem~\ref{thm:main}(i).

For part~(ii), fix a compact neighborhood $K\subset\Xi$ with
$\xi_0\in\operatorname{int}K$. Since $DH(\xi_0)$ has rank two, it has a
nonzero $2\times2$ minor. Let $e_i,e_j$ be the coordinate vectors associated
with the columns of this minor. Since $\xi_0\in\operatorname{int}K$ and $H$
is continuously differentiable near $\xi_0$, there is an open neighborhood
$O_0\subset\R^2$ of $(0,0)$ such that
$\xi_0+se_i+te_j\in K$ for every $(s,t)\in O_0$. Define
\[
 \widehat H:O_0\to\R^2,
 \qquad
 \widehat H(s,t)=H(\xi_0+se_i+te_j).
\]
The map $\widehat H$ is continuously differentiable, and its derivative at
$(0,0)$ is the chosen minor and is therefore invertible. By the inverse
function theorem
\cite[Theorem~9.24]{Rudin1976}, there are open neighborhoods $O\subset O_0$
of $(0,0)$ and $V$ of $H(\xi_0)$ such that $\widehat H$ maps $O$
diffeomorphically onto $V$.
Consequently, $V\subset H(K)$.
Because $B(\xi_0)\ge3\sqrt3/2$, there is a point
$p_1=(a_1,b_1)\in V$ with $b_1>3\sqrt3/2$. Choose a compact rectangle
$Q\subset V$ with $p_1\in\operatorname{int}Q$. Theorem~\ref{thm:main}(ii)
shows that $Q$ is not robustly stabilizable. If $\mu$ stabilized
\eqref{eq:general-parameter-map} over $K$, then $\widetilde\mu$ would
stabilize \eqref{eq:system} over $H(K)$ and hence over $Q$. This is a
contradiction.

For part~(iii), choose a nondegenerate compact interval $J$ containing
$A(\Xi)$.
Then $H(\Xi)\subset J\times B(\Xi)$. Since $B(\Xi)$ is finite and
$\sup_{\xi\in\Xi}B(\xi)<4$, Theorem~\ref{thm:finite-set} shows that
$J\times B(\Xi)$, and hence its subset $H(\Xi)$, is robustly stabilizable.
\bibliographystyle{plain}
\bibliography{references}

\end{document}